\documentclass[11pt,reqno]{amsart}
\usepackage[T1]{fontenc}
\usepackage{lmodern,microtype,amsmath,amssymb,mathtools}
\usepackage[margin=28mm]{geometry}
\usepackage{url}

\newtheorem{theorem}{Theorem}[section]
\newtheorem{proposition}[theorem]{Proposition}
\newtheorem{lemma}[theorem]{Lemma}
\newtheorem{corollary}[theorem]{Corollary}
\theoremstyle{definition}
\newtheorem{definition}[theorem]{Definition}
\theoremstyle{remark}
\newtheorem{remark}[theorem]{Remark}
\numberwithin{equation}{section}

\newcommand{\T}{\mathbb{T}}
\newcommand{\D}{\mathbb{D}}
\newcommand{\R}{\mathbb{R}}
\newcommand{\C}{\mathbb{C}}
\newcommand{\Z}{\mathbb{Z}}
\newcommand{\TT}{\mathcal{T}}
\newcommand{\FF}{\mathcal{F}}
\newcommand{\Q}{\mathcal{Q}}
\newcommand{\eps}{\varepsilon}

\newcommand{\SL}{\operatorname{SL}}
\newcommand{\SU}{\operatorname{SU}}
\newcommand{\SO}{\operatorname{SO}}
\newcommand{\Sp}{\operatorname{Sp}}
\newcommand{\Sym}{\operatorname{Sym}}
\DeclareMathOperator{\tr}{tr}
\DeclareMathOperator{\wind}{wind}
\DeclareMathOperator{\diag}{diag}
\DeclareMathOperator{\Log}{Log}
\DeclareMathOperator{\dist}{dist}

\allowdisplaybreaks

\title[Minimal acceleration and premonotonic cocycles]{Minimal acceleration and the topology of premonotonic cocycles}

\author{Xu Xia}
\address{(Xu Xia) State Key Laboratory of Mathematical Sciences,
Academy of Mathematics and Systems Science,
Chinese Academy of Sciences, 100190 Beijing, P. R. China}
\email{xiaxu14@mails.ucas.ac.cn}

\author[Jiahao Xu]{Jiahao Xu}
\address{(Jiahao Xu) Beijing International Center for Mathematical Research,
Peking University, Beijing 100871, China}
\email{xjhao512@outlook.com}
\thanks{Corresponding author: Jiahao Xu. X.~Xia was partially supported by the National Natural Science Foundation of China (Grant No.~12301218).}

\date{October 18, 2026}

\makeatletter
\@ifundefined{subjclassname@2020}{\@namedef{subjclassname@2020}{\textup{2020} Mathematics Subject Classification}}{}
\makeatother

\begin{document}

\begin{abstract}
In his global theory of one-frequency Schr\"odinger
operators~\cite{Av15}, Avila asked whether the acceleration of
a real-analytic $\SL(2,\R)$ cocycle over an irrational circle
rotation is bounded below by the absolute value of its
topological degree and whether equality in nonzero degree
characterizes premonotonicity. We answer this question
affirmatively for every irrational frequency, without
arithmetic restrictions. Here premonotonicity means that some
iterate becomes strictly monotonic in the base variable after
a periodic change of coordinates~\cite{AK15}. In the positive
Lyapunov exponent regime, this provides a nonzero-degree
counterpart to Avila's characterization of uniform hyperbolicity
by vanishing acceleration.

In nonzero degree, we also characterize minimal acceleration
through signed connections and invariant positive matrix-valued
measures. In the positive-degree, positive Lyapunov exponent case,
the proof constructs an analytic positive connection using
invariant positive matrix-valued measures and a transfer--scattering
representation. A complex-analytic construction then produces a
periodic gauge. As a further consequence, for each fixed irrational
frequency and nonzero degree, the premonotonic locus is path connected:
every cocycle in this locus can be joined to the linear rotation of
the same degree by a path preserving premonotonicity.
\end{abstract}

\maketitle

\section{Introduction}

Let $\T=\R/\Z$ and fix an irrational frequency $\alpha\in\R\setminus\mathbb Q$. A real-analytic map $A:\T\to\SL(2,\R)$ defines a quasiperiodic cocycle $(\alpha,A)$ on $\T\times\R^2$ by
\[
(x,v)\longmapsto (x+\alpha,A(x)v).
\]
Its $n$-step transfer matrices and top Lyapunov exponent are given by
\[
A_n(x)=A(x+(n-1)\alpha)\cdots A(x),\qquad
L(\alpha,A)=\lim_{n\to\infty}\frac1n\int_\T\log\|A_n(x)\|\,dx,
\]
where $\|\cdot\|$ is the Euclidean operator norm and $dx$ is normalized Lebesgue measure on $\T$. For a continuous matrix field $K$ on $\T$, write $\|K\|_\infty=\sup_{x\in\T}\|K(x)\|$; spaces of continuous matrix fields carry this norm. For a holomorphic extension to $|\Im z|<h$, write $A_y(x)=A(x+iy)$. Avila introduced the acceleration \cite{Av15}
\[
\omega(\alpha,A)=\frac1{2\pi}\lim_{y\downarrow0}
\frac{L(\alpha,A_y)-L(\alpha,A)}{y},
\]
which is well defined by convexity and takes integer values. For real cocycles with $L(\alpha,A)>0$, Avila's regularity theorem \cite[Theorems~5 and~6]{Av15} identifies vanishing acceleration with uniform hyperbolicity. Indeed, $y\mapsto L(\alpha,A_y)$ is even, and quantization makes zero acceleration equivalent to its being affine near $y=0$.

For Schr\"odinger cocycles, acceleration can be described through zero counts and dual dynamics. Building on zero-counting methods of Goldstein and Schlag \cite{GS08}, Han and Schlag \cite{HS26} related determinant zero counts to acceleration for analytic Schr\"odinger cocycles with positive Lyapunov exponent under a Diophantine condition on the frequency. Ge, Jitomirskaya, You, and Zhou \cite{GJYZ26} expressed the complexified Lyapunov exponent in terms of dual Lyapunov exponents through a multiplicative Jensen formula; for trigonometric polynomial potentials at real energies, positive acceleration is half the dimension of the center of the partially hyperbolic dual cocycle. These spectral descriptions concern degree-zero cocycles. Here we study how minimal acceleration determines the real dynamics in nonzero degree.

The homotopy degree $d=\deg A\in\Z$ is induced by the deformation retraction of $\SL(2,\R)$ onto $\SO(2,\R)$. We use the convention
\[
J=\begin{pmatrix}0&-1\\1&0\end{pmatrix},\qquad
R_\theta=\exp(2\pi\theta J),\qquad \deg(x\mapsto R_{dx})=d.
\]
Nonzero degree rules out continuous real invariant line bundles and hence uniform hyperbolicity (Lemma~\ref{lem62ldd}). It also imposes the lower bound
\[
L(\alpha,A_y)\ge L(\alpha,A)+2\pi|d|y,\qquad \text{and hence}\qquad \omega(\alpha,A)\ge|d|
\]
for $0\le y<h$ (Proposition~\ref{prop21lww}). We ask which cocycles attain this lower bound.

The relevant notion is premonotonicity, introduced by Avila and Krikorian \cite{AK15}. For $A\in C^1(\T,\SL(2,\R))$, let $Q_A(x)=-A(x)^T J A'(x)$. Then $Q_A(x)$ is symmetric and gives the derivative of the vector angle:
\begin{equation}\label{eq13angg}
\frac{d}{dx}\arg(A(x)v)=\frac{v^TQ_A(x)v}{\|A(x)v\|^2},\qquad v\in\R^2\setminus\{0\}.
\end{equation}
The map $A$ is called \emph{strictly positively monotonic} if $Q_A(x)>0$ for all $x\in\T$, \emph{strictly negatively monotonic} if $Q_A(x)<0$ for all $x\in\T$, and \emph{strictly monotonic} if either holds.

\begin{definition}[\cite{AK15}, Section~3.4]
A cocycle $(\alpha,A)$ is \emph{premonotonic} if there exist $N\ge1$ and a 1-periodic gauge $B\in C^1(\T,\SL(2,\R))$ such that the conjugated iterate
\[
x\longmapsto B(x+N\alpha)A_N(x)B(x)^{-1}
\]
is strictly monotonic on $\T$.
\end{definition}

For analytic $A$, the gauge can be chosen analytic by $C^1$ approximation. By Lemma~\ref{lem21evmm}, the same gauge makes all sufficiently long iterates strictly monotonic.

Avila and Krikorian \cite{AK15} developed a theory of premonotonic cocycles that applies without arithmetic restrictions on the irrational frequency. For a fixed frequency, this class is $C^1$-open and contains all nonzero-degree $\SO(2,\R)$-valued cocycles. They proved that an analytic premonotonic cocycle with zero Lyapunov exponent is analytically conjugate to an $\SO(2,\R)$-valued cocycle. They also established analytic dependence of the Lyapunov exponent along one-parameter analytic families of analytic premonotonic cocycles at a fixed frequency.

Avila noted that premonotonic cocycles satisfy $\omega(\alpha,A)=|\deg A|$ \cite[p.~17, item~(3)]{Av15}. He asked whether $|\deg A|$ always bounds the acceleration from below and whether, in nonzero degree, equality implies premonotonicity \cite[p.~18, item~(4)]{Av15}. We prove the lower bound and the converse characterization for every irrational frequency.

\begin{theorem}\label{thm11mnn}
Let $\alpha\in\R\setminus\mathbb Q$ and $A\in C^\omega(\T,\SL(2,\R))$ with $d=\deg A\ne0$. Then
\[
\omega(\alpha,A)=|d|
\quad\Longleftrightarrow\quad
(\alpha,A)\text{ is premonotonic}.
\]
\end{theorem}

When $L(\alpha,A)>0$, the theorem characterizes premonotonicity among nonuniformly hyperbolic cocycles of nonzero degree by the vanishing of $\omega(\alpha,A)-|d|$. Section~\ref{sec9mchh} compares this with Avila's degree-zero criterion for uniform hyperbolicity. Related results concern acceleration and dynamical degree for compact $\SO(s,\R)$ cocycles \cite[Theorem~B]{HPZ24}, and harmonicity of the Lyapunov exponent for monotonic cocycles \cite{Tak25}.

Minimal acceleration also admits an infinitesimal characterization. Write $\Sym_2\R$ for the space of real symmetric $2\times2$ matrices. We call a 1-periodic field $H:\T\to\Sym_2\R$ a \emph{positive connection} for $(\alpha,A)$ if
\begin{equation}\label{eq104psii}
Q_A(x)+A(x)^TH(x+\alpha)A(x)-H(x)>0\qquad (x\in\T).
\end{equation}
This is the positive form of Avila's signedness condition \cite[eq.~(35) and Remark~17]{Av15}. Writing $T_\alpha x=x+\alpha$, Hahn--Banach separation expresses the obstruction to the existence of a continuous positive connection in terms of nonzero finite $\Sym_2\R$-valued Borel measures $\mu$ that are positive, meaning $\mu(E)$ is positive semidefinite for every Borel set $E\subset\T$, and satisfy
\[
\mu(T_\alpha E)=\int_E A(x)\,d\mu(x)\,A(x)^T
\quad\text{for all Borel sets }E\subset\T.
\]
Integration against matrix-valued measures is understood entrywise.

\begin{proposition}\label{prop103crr}
Let $\alpha\in\R\setminus\mathbb Q$ and $A\in C^\omega(\T,\SL(2,\R))$ with $d=\deg A>0$. The following are equivalent:
\begin{enumerate}
\item $\omega(\alpha,A)=d$;
\item there exists a real-analytic map $H:\T\to\Sym_2\R$ satisfying \eqref{eq104psii};
\item every nonzero finite invariant positive matrix-valued measure $\mu$ satisfies $\int_\T\tr(Q_A\,d\mu)>0$.
\end{enumerate}
For $d<0$, the corresponding equivalence holds with $\omega(\alpha,A)=-d=|d|$ in~(1) and with the inequalities in~(2) and~(3) reversed.
\end{proposition}

In nonzero degree these invariant matrix measures always exist (Proposition~\ref{prop902ehh}). The dual variational formula is given in Proposition~\ref{prop701mgg}.

As a consequence of the characterization, we obtain path connectedness of the minimal-acceleration locus
\[
\mathcal M_d(\alpha)=\bigl\{A\in C^\omega(\T,\SL(2,\R)):\deg A=d,\ \omega(\alpha,A)=|d|\bigr\}
\]
for fixed $\alpha\in\R\setminus\mathbb Q$ and $d\ne0$.

\begin{theorem}\label{thm14ptth}
For every $\alpha\in\R\setminus\mathbb Q$ and $d\in\Z\setminus\{0\}$, any $A\in\mathcal M_d(\alpha)$ can be joined to the linear rotation $x\mapsto R_{dx}$ by a path entirely contained in $\mathcal M_d(\alpha)$ that is piecewise real analytic in the path parameter and continuous in the supremum norm on a common complex strip $|\Im z|\le h$ for some $h>0$. In particular, $\mathcal M_d(\alpha)$ is path connected, as is the degree-$d$ premonotonic locus in $C^r(\T,\SL(2,\R))$ for every integer $r\ge1$ (Corollary~\ref{cor803crr}).
\end{theorem}
Here the width $h>0$ of the common strip may be smaller than the initial analyticity width of $A$.

Related deformation results include approximation by linear rotations for cocycles already conjugate to rotations~\cite[Lemma~3.10]{AK15}, connectedness and accessibility at the $C^0$ level~\cite[Theorem~4 and Appendix~A]{ABD12}, and homotopies of positive symplectic loops~\cite{Sli01,WZ24}.

\subsection*{Main ideas of the proof}
An elementary model is given by the finite Fourier case in Appendix~\ref{app10frr}. Suppose that $A$ has Fourier support in $[-m,m]$ with $m\ge1$ and $\deg A=m$. For each $v\in\R^2\setminus\{0\}$, set $f_v(x)=(A(x)v)_1+i(A(x)v)_2$ and $z=e^{2\pi ix}$. The argument principle shows that the polynomial $z^m f_v(x)$ has exactly $2m$ zeros $a_j$, counted with multiplicity, all in the unit disk $\D=\{\zeta\in\C:|\zeta|<1\}$, and
\[
\frac{d}{dx}\arg f_v(x)
=\pi\sum_{j=1}^{2m}\frac{1-|a_j|^2}{|e^{2\pi ix}-a_j|^2}>0.
\]
Each summand is a positive Poisson kernel, so $A$ itself is strictly positively monotonic.

For general analytic cocycles, the proof obtains angular positivity from contraction in a complex strip. Let $\eps>0$, and let $g$ be 1-periodic and holomorphic on a neighborhood of the closed strip $0\le\Im z\le\eps$, with $|g(x)|=1$ and $|g(x+i\eps)|\le r$ for some $0<r<1$. The minimum principle applied to the superharmonic function $-\log|g|$, with value $+\infty$ at zeros, and the Cauchy--Riemann equations give
\[
\begin{aligned}
&-\log|g(x+iy)|\ge\frac{y}{\eps}\log\frac1r,~\forall~0\le y\le\eps,\\
&\partial_x\arg g(x)
=-\left.\partial_y\log|g(x+iy)|\right|_{y=0}
\ge\frac1\eps\log\frac1r>0.
\end{aligned}
\]
Thus uniform inward motion after complexification forces uniformly positive angular motion on the real circle. Lemma~\ref{lem601ohh} applies this principle to the projective action of long transfer products in suitable Cayley coordinates.

For $d>0$, $L(\alpha,A)>0$, and $\omega(\alpha,A)=d$, we construct a 1-periodic analytic gauge in which, at a positive height, the complex unstable direction lies in $\D$ and the stable direction outside $\overline\D$. First, we obtain a positive connection satisfying \eqref{eq104psii}. Hahn--Banach separation reduces this to positivity of the pairing against invariant positive matrix measures, which are singular and have rank-one density. The transfer--scattering correspondence \cite{MSB13} realizes the pairing as a boundary phase derivative in a covariant unitary representation. An exponential dichotomy gives a resolvent bound independent of the representation, and the three-lines theorem makes the derivative strictly positive. Integrating this connection gives a strictly monotonic cocycle on the universal cover.

Second, we adjust the monodromy of this moving frame. The integrated gauge need not be periodic: after one circuit of the base circle, it may return multiplied by a fixed matrix. Avila and Krikorian discussed premonotonicity in terms of transverse projective foliations and holonomy in their early manuscript~\cite[Appendix~F]{AKdraft}. Passing from the local interpretation of signedness \cite[Remark~17]{Av15} to a periodic frame requires control of this monodromy; related holonomy issues appear in \cite[Remark~4.11]{LXZ26}. Interpolating between positive connections lets us choose the monodromy to be a quarter-turn. Third, we cancel that turn with a holomorphic rotation built from a Blaschke factor. Its modulus is chosen to preserve the separation of the unstable and stable directions by $\partial\D$ at a positive height, while the corrected gauge becomes 1-periodic.

In these periodic coordinates, minimal acceleration makes the winding numbers of a diagonal entry of sufficiently long transfer products agree on the real and upper circles. The argument principle excludes zeros of that entry between them. Uniform hyperbolicity and the separation of the invariant directions then give contraction of the projective disk at the upper boundary. The scalar boundary estimate above gives a positive derivative of the Cayley angle. Since this angle is twice the real projective angle, sufficiently long iterates are strictly positively monotonic on the real circle.

When $L(\alpha,A)=0$, minimal acceleration makes the Lyapunov exponent locally analytic as a function of the cocycle. Nonnegativity then makes $A$ a stationary point. The rigidity argument of \cite[Remark~13]{Av15} gives an analytic invariant metric, from which unique ergodicity yields premonotonicity. For path connectedness, we construct an $\SO(2,\R)$-equivariant contraction preserving strict monotonicity and the quarter-turn symmetry. It therefore descends to the original circle without changing the frequency. Negative degree follows by reflection.

\subsection*{Organization of the paper}
Section~\ref{sec2prr} proves the degree bound and treats zero Lyapunov exponent; Sections~\ref{sec301cnn}--\ref{sec6fnn} prove the positive-exponent characterization. Sections~\ref{sec70vrr}--\ref{sec9mchh} give the variational formula, path connectedness, and the degree-zero comparison. Appendix~\ref{app10frr} treats finite Fourier extremality.

\section{Complexification, degree lower bound, and zero Lyapunov exponent}\label{sec2prr}

Throughout the paper, we extend real-analytic maps on $\T$ holomorphically to strips $|\Im z|<h$, shrinking $h>0$ as needed. The extension of an $\SL(2,\R)$-valued map takes values in $\SL(2,\C)$. For a fixed frequency, we abbreviate the base rotation $T_\alpha$ to $T$.

\subsection{Persistence of monotonicity under iteration}

We first record the elementary fact that strict monotonicity of one iterate implies strict monotonicity of all sufficiently long iterates.

\begin{lemma}\label{lem21evmm}
Let $\alpha\in\R$, $C\in C^1(\T,\SL(2,\R))$, and $\sigma\in\{1,-1\}$.
If $\sigma Q_{C_m}>0$ on $\T$ for some $m\ge1$, then $\sigma Q_{C_n}>0$ on $\T$ for all sufficiently large $n$. In particular, for $A\in C^1(\T,\SL(2,\R))$, if a 1-periodic gauge $B\in C^1(\T,\SL(2,\R))$ makes $(B\circ T^m)A_mB^{-1}$ strictly monotonic for some $m\ge1$, then $(B\circ T^n)A_nB^{-1}$ is strictly monotonic for all large $n$.
\end{lemma}
\begin{proof}
Set $C_0=I$. Differentiating $C_{a+b}(x)=C_b(x+a\alpha)C_a(x)$ and using $C_b^T J C_b=J$ gives the cocycle identity
\begin{equation}\label{eq21pqq}
Q_{C_{a+b}}(x)=Q_{C_a}(x)+C_a(x)^TQ_{C_b}(x+a\alpha)C_a(x).
\end{equation}
Choose $\delta>0$ such that $\sigma Q_{C_m}\ge\delta I$ on $\T$, and set
\[
K=\max_{0\le r<m}\|Q_{C_r}\|_\infty,\qquad
c=\delta\min_{0\le r<m}\|C_r^{-1}\|_\infty^{-2}>0.
\]
Fix $x\in\T$, a unit vector $v\in\R^2$, and $0\le r<m$. Applying \eqref{eq21pqq} to the decompositions $C_{jm+r}(x)=C_r(x+jm\alpha)C_{jm}(x)$ and $C_{(j+1)m+r}(x)=C_m(x+(jm+r)\alpha)C_{jm+r}(x)$ yields, for every $j\ge0$,
\[
\sigma v^TQ_{C_{jm+r}}(x)v \ge \delta\sum_{\ell=0}^{j-1}\|C_{\ell m}(x)v\|^2 - K\|C_{jm}(x)v\|^2
\]
and
\[
\sigma v^TQ_{C_{(j+1)m+r}}(x)v - \sigma v^TQ_{C_{jm+r}}(x)v
\ge \delta\|C_{jm+r}(x)v\|^2 \ge c\|C_{jm}(x)v\|^2.
\]
If $K=0$, the first inequality immediately gives $\sigma Q_{C_n}>0$ for all $n\ge m$. Suppose $K>0$ and $\sigma v^TQ_{C_{km+r}}(x)v\le0$ for some $k\ge1$. Since $j\mapsto\sigma v^TQ_{C_{jm+r}}(x)v$ is strictly increasing, it is nonpositive for all $0\le j\le k$. Summing the second inequality over $0\le j<k$ gives
\[
c\sum_{j=0}^{k-1}\|C_{jm}(x)v\|^2 \le \sigma v^TQ_{C_{km+r}}(x)v - \sigma v^TQ_{C_r}(x)v \le K.
\]
On the other hand, nonpositivity in the first inequality implies
\[
\sum_{\ell=0}^j\|C_{\ell m}(x)v\|^2\ge\left(1+\frac{\delta}{K}\right)\sum_{\ell=0}^{j-1}\|C_{\ell m}(x)v\|^2, \quad 1\le j<k,
\]
which by induction from $\|C_0(x)v\|^2=1$ gives $\sum_{j=0}^{k-1}\|C_{jm}(x)v\|^2\ge(1+\delta/K)^{k-1}$. These two bounds are incompatible once $k$ exceeds a threshold depending only on $\delta,K,c$. The final assertion follows from $C_n=(B\circ T^n)A_nB^{-1}$ when $C=(B\circ T)AB^{-1}$.
\end{proof}

\subsection{Cayley coordinates and the degree lower bound}

To pass between $\SL(2,\R)$ and $\SU(1,1)$, we conjugate by the unitary matrix
\begin{equation}\label{eq202cyy}
\Q=\frac1{\sqrt2}\begin{pmatrix}1&-i\\1&i\end{pmatrix},\qquad
\widehat A=\Q A\Q^*=\begin{pmatrix}a&b\\c&e\end{pmatrix}.
\end{equation}
We use the same Cayley transform $\widehat D=\Q D\Q^*$ for any $2\times2$ matrix-valued map $D$, and denote the entries of $\widehat D_n$ by $a_n,b_n,c_n,e_n$ in the order of~\eqref{eq202cyy}. For $x\in\T$, the matrix $\widehat A(x)$ belongs to $\SU(1,1)$, so that $c=\bar b$, $e=\bar a$, and $|a|^2-|b|^2=1$. Writing $\zeta=w/u$ for the projective coordinate of $(u,w)^T\in\C^2$, the induced M\"obius action on $\widehat\C=\C\mathbb P^1$ is
\begin{equation}\label{eq23mbb}
\widehat A\cdot\zeta=\frac{c+e\zeta}{a+b\zeta},
\end{equation}
and the Cayley coordinates identify the real projective line $\R\mathbb P^1$ with the unit circle $\partial\D$. Under this identification, $R_\theta$ acts by multiplication by $e^{4\pi i\theta}$, so a loop $A:\T\to\SL(2,\R)$ of degree $d$ has projective degree $2d$. Since $|a(x)|^2=1+|b(x)|^2\ge1$ on $\T$, $a$ is nonvanishing on $\T$ and homotopy to $R_{dx}$ gives
\[
\wind a=-d,
\]
where $\wind g=(2\pi i)^{-1}\int_\T(g'/g)\,dx$ denotes the winding number of a nonvanishing $C^1$ loop $g:\T\to\C^\times$.

\begin{proposition}\label{prop21lww}
For every real-analytic cocycle $(\alpha,A)$ over an irrational rotation $\alpha\in\R\setminus\mathbb Q$,
\[
L(\alpha,A_y)\ge L(\alpha,A)+2\pi|\deg A|y \quad (0\le y<h), \qquad \omega(\alpha,A)\ge|\deg A|.
\]
\end{proposition}
\begin{proof}
Assume first that $d=\deg A\ge0$, and let $a_n(z)$ be the upper-left entry of $\widehat A_n(z)$. Since $z\mapsto\log|a_n(z)|$ is subharmonic and 1-periodic on $|\Im z|<h$, its horizontal mean
\[
u_n(y)=\int_\T\log|a_n(x+iy)|\,dx
\]
is convex on the full interval $(-h,h)$. On $\T$, $a_n$ is nonvanishing with $\wind a_n=-nd$. By compactness of $\T$, $a_n$ has a zero-free collar $|y|<\delta_n$ on which $\log|a_n|$ is harmonic and the Cauchy--Riemann equations give
\[
\left.\partial_y\log|a_n(x+iy)|\right|_{y=0}=-\partial_x\arg a_n(x),\qquad
u_n'(0)=-2\pi\wind a_n=2\pi nd.
\]
Convexity of $u_n$ on $(-h,h)$ therefore implies $u_n(y)\ge u_n(0)+2\pi ndy$ for all $y\in[0,h)$. On $\T$, the $\SU(1,1)$ relation $|a_n|^2-|b_n|^2=1$ gives
\[
|a_n(x)|\le\|A_n(x)\|=|a_n(x)|+|b_n(x)|\le2|a_n(x)|,
\]
while for $y\in(0,h)$ we always have $|a_n(x+iy)|\le\|A_n(x+iy)\|$. Dividing $u_n(y)\ge u_n(0)+2\pi ndy$ by $n$ and letting $n\to\infty$ yields $L(\alpha,A_y)\ge L(\alpha,A)+2\pi dy$. The case $d<0$ follows symmetrically from the lower-right entry $e_n$ (for which $\wind e_n=nd$), and differentiating at $y=0^+$ gives $\omega(\alpha,A)\ge|d|$.
\end{proof}

\subsection{Uniform hyperbolicity and disk contraction}\label{sec23hyp}

Recall that a complex cocycle $(\alpha,D)$ with $D\in C^\omega(\T,\SL(2,\C))$ is uniformly hyperbolic if $\C^2$ splits into continuous $D$-invariant line bundles $E^u\oplus E^s$ exhibiting uniform exponential expansion and contraction. For analytic $D$, these invariant bundles are analytic. Approximating continuous nonvanishing sections of $E^u$ and $E^s$ on $\T$ by trigonometric polynomials, projecting back onto $E^u,E^s$, and normalizing the determinant yields a holomorphic 1-periodic frame $V:\T\to\SL(2,\C)$ on a thinner strip that diagonalizes $D$. If $\lambda:\T\to\C^\times$ denotes the resulting unstable multiplier, then by~\cite{Av15},
\begin{equation}\label{eq206wdll}
L(\alpha,D)=\int_\T\log|\lambda(x)|\,dx,
\qquad \omega(\alpha,D)=-\wind\lambda.
\end{equation}
For arbitrary analytic cocycles, both $L$ and $\omega$ are invariant under holomorphic 1-periodic $\SL(2,\C)$ conjugacies and satisfy the iteration identities
\begin{equation}\label{eq27itrr}
L(n\alpha,D_n)=nL(\alpha,D),\qquad \omega(n\alpha,D_n)=n\omega(\alpha,D), \quad n\ge1.
\end{equation}

For the 4-periodic gauges used below, we fix the covering normalization. Let $A:\T\to\SL(2,\R)$ and $B_0:\R/4\Z\to\SL(2,\R)$ be real analytic, and set $C=(B_0\circ T)AB_0^{-1}$. Write $L^{(4)}(\alpha,C_y)$ for the Lyapunov exponent on $\R/4\Z$ computed with normalized Haar measure $dx/4$. Boundedness of $B_0^{\pm1}$ on smaller closed strips gives $L^{(4)}(\alpha,C_y)=L(\alpha,A_y)$. For $\widetilde C(u)=C(4u)$, changing variables $x=4u$ yields, for all sufficiently small $|y|$,
\[
\begin{aligned}
L(\alpha/4,\widetilde C_y)&=L^{(4)}(\alpha,C_{4y})=L(\alpha,A_{4y}),\\
\omega(\alpha/4,\widetilde C)&=4\omega(\alpha,A).
\end{aligned}
\]
The 1-periodic gauge $B_0(4\cdot)$ conjugates $A(4\cdot)$ to $\widetilde C$, so $\deg\widetilde C=4\deg A$.

Suppose now that $A\in C^\omega(\T,\SL(2,\R))$ has $d=\deg A>0$ and attains minimal acceleration $\omega(\alpha,A)=d$. By Avila's quantization theorem \cite[Theorem~5]{Av15}, the right derivative of the convex function $y\mapsto L(\alpha,A_y)$ lies in $2\pi\Z$ at every $y$. Since the right derivative of a convex function is nondecreasing and right-continuous, there exists $\eta\in(0,h/2)$ such that $y\mapsto L(\alpha,A_y)$ is affine on $[0,2\eta]$. Thus
\begin{equation}\label{eq208aff}
L(\alpha,A_y)=L(\alpha,A)+2\pi dy,\qquad 0\le y\le\eta.
\end{equation}
For every $y\in(0,\eta]$, this implies $L(\alpha,A_y)>0$ and that $t\mapsto L(\alpha,A_{y+t})$ is affine on a neighborhood of $t=0$. Avila's regularity theorem \cite[Theorem~6]{Av15} therefore applies to $A_y$, showing that $(\alpha,A_y)$ is uniformly hyperbolic for every $y\in(0,\eta]$, regardless of whether $L(\alpha,A)$ is zero or positive.

One standard way to obtain uniform hyperbolicity is strict disk contraction in Cayley coordinates: if a holomorphic family $C$ maps $\overline\D$ strictly into $\D$ under the M\"obius action~\eqref{eq23mbb}, uniform contraction of the Poincar\'e metric on $\D$ produces holomorphic invariant unstable and stable sections $m_u:\T\to\D$ and $m_s:\T\to\widehat\C\setminus\overline\D$, making the cocycle uniformly hyperbolic. For strictly positively monotonic real cocycles, this contraction holds at every sufficiently small positive height: if $Q_C>0$ on $\T$, then for any $|\zeta|=1$, writing $\widehat C(x)\cdot\zeta=e^{i\vartheta(x)}$ locally gives $\vartheta'(x)>0$ by \eqref{eq13angg}, so the Cauchy--Riemann equations yield
\[
\left.\partial_y\log|\widehat C(x+iy)\cdot\zeta|\right|_{y=0}=-\vartheta'(x)<0.
\]
Writing the upper row of $\widehat C$ as $(a_C,b_C)$, the $\SU(1,1)$ relations and compactness give $|a_C(z)|>|b_C(z)|$ on a sufficiently thin strip. Thus $a_C(z)+b_C(z)\zeta\ne0$ for $\zeta\in\overline\D$. Compactness makes the inward motion of the boundary uniform, and the maximum principle in $\zeta$ then shows that $\widehat C(x+iy)$ maps $\overline\D$ strictly into $\D$ for all sufficiently small $y>0$.

\subsection{The zero-exponent case}\label{sec24zro}

We now prove the forward implication in Theorem~\ref{thm11mnn} when $L(\alpha,A)=0$. Since the Lyapunov exponent is nonnegative, local analyticity makes a zero-exponent cocycle a stationary point. Fix a strip width $h>0$ smaller than the analyticity width of $A$, and work in the real Banach manifold $\mathcal B_h$ of holomorphic 1-periodic maps $\widetilde A:\{|\Im z|<h\}\to\SL(2,\C)$ that extend continuously to $|\Im z|\le h$ and take values in $\SL(2,\R)$ on $\T$. We equip $\mathcal B_h$ with the topology induced by the supremum norm on the closed strip.

\begin{proposition}\label{prop23allc}
Let $A\in\mathcal B_h$ have $d=\deg A\ne0$ and $\omega(\alpha,A)=|d|$. Then there is a neighborhood $\mathcal U\subset\mathcal B_h$ of $A$ on which every $\widetilde A\in\mathcal U$ has degree $d$ and acceleration $|d|$, and the functional $\widetilde A\mapsto L(\alpha,\widetilde A)$ is real analytic on $\mathcal U$.
\end{proposition}
\begin{proof}
By the argument leading to~\eqref{eq208aff}, with $d$ replaced by $|d|$, choose $\eta\in(0,h)$ such that $A_\eta$ is uniformly hyperbolic with acceleration $|d|$. Since uniform hyperbolicity is open and the winding number of the unstable multiplier is locally constant, there is a neighborhood $\mathcal U\subset\mathcal B_h$ of $A$ such that for every $\widetilde A\in\mathcal U$, $\deg\widetilde A=d$ and $\widetilde A_\eta$ is uniformly hyperbolic with $\omega(\alpha,\widetilde A_\eta)=|d|$. By Proposition~\ref{prop21lww}, $\omega(\alpha,\widetilde A)\ge|d|$. Since the convex function $y\mapsto L(\alpha,\widetilde A_y)$ has right derivative $\ge 2\pi|d|$ at $y=0$ and derivative $2\pi|d|$ at $y=\eta$, its derivative is identically $2\pi|d|$ on $[0,\eta]$. Hence $\omega(\alpha,\widetilde A)=|d|$ and
\begin{equation}\label{eq29abb}
L(\alpha,\widetilde A)=L(\alpha,\widetilde A_\eta)-2\pi|d|\eta, \quad \widetilde A\in\mathcal U.
\end{equation}
After shrinking $\mathcal U$ if needed, the holomorphic invariant frames and the unstable multiplier $\lambda_{\widetilde A}$ of $\widetilde A_\eta$ depend holomorphically on complex perturbations of $\widetilde A$ (by the contracting projective graph transform), with $\lambda_{\widetilde A}/\lambda_A$ uniformly close to $1$ on $\T$. Therefore
\[
L(\alpha,\widetilde A_\eta)-L(\alpha,A_\eta)
=\Re\int_\T\log\frac{\lambda_{\widetilde A}(x)}{\lambda_A(x)}\,dx
\]
is pluriharmonic on the complexified neighborhood and real analytic on $\mathcal U$ (compare \cite[Proposition~5]{Av15}). By~\eqref{eq29abb}, $L(\alpha,\cdot)$ is real analytic on $\mathcal U$.
\end{proof}

Next, we give a self-contained proof of the rigidity result for stationary points stated by Avila in \cite[Remark~13, pp.~27--28]{Av15}.

\begin{proposition}\label{prop204staa}
Suppose $A\in\mathcal B_h$ is such that $A_y$ is uniformly hyperbolic for all $0<y\le\eta$, where $\eta\in(0,h)$. If
\[
\left.\frac d{dt}\right|_{t=0}
L\bigl(\alpha,(A\exp(tW))_\eta\bigr)=0
\]
for every 1-periodic $W\in C^\omega(\T,\mathfrak{sl}(2,\R))$ extending to $|\Im z|\le h$, then there exists a 1-periodic gauge $B\in C^\omega(\T,\SL(2,\R))$ such that $(B\circ T)AB^{-1}$ takes values in $\SO(2,\R)$.
\end{proposition}
\begin{proof}
For $0<\Im z\le\eta$, let $P(z)$ be the holomorphic projection onto the unstable direction of $A(z)$ along the stable direction, so that $P^2=P$, $\tr P=1$, and $P(z+\alpha)=A(z)P(z)A(z)^{-1}$. Set $K(z)=P(z)-\frac12 I$. We first verify the first-variation formula
\begin{equation}\label{eq211pjdv}
\left.\frac d{dt}\right|_{t=0}L\bigl(\alpha,(A\exp(tW))_\eta\bigr)
=\Re\int_\T\tr\bigl(K(x+i\eta)W(x+i\eta)\bigr)\,dx.
\end{equation}
Choose a family of holomorphic 1-periodic frames $V_t$ near height $\eta$, depending analytically on $t$ near $0$, such that $A\exp(tW)V_t=(V_t\circ T)\diag(\lambda_t,\lambda_t^{-1})$. Differentiating the $(1,1)$-entry at $t=0$ with $k=(V_0^{-1}\dot V_0)_{11}$ gives
\[
\frac{\dot\lambda_0}{\lambda_0}=(V_0^{-1}WV_0)_{11}+k-k\circ T
=\tr(PW)+k-k\circ T.
\]
Integrating over $\T$, taking real parts, and using $\tr W=0$ yields \eqref{eq211pjdv}.

By Cauchy's theorem, the Fourier coefficients
\[
K_j=e^{2\pi j y}\int_\T K(x+iy)e^{-2\pi i jx}\,dx, \quad j\in\Z,
\]
are independent of $y\in(0,\eta]$. We test~\eqref{eq211pjdv} against $W(x)=U\cos(2\pi jx)$ and $W(x)=U\sin(2\pi jx)$ for arbitrary $U\in\mathfrak{sl}(2,\R)$ and $j\ge0$. The factors $e^{\pm 2\pi j\eta}$ cancel in the pairing, leaving
\[
\Re\tr\bigl((K_j+K_{-j})U\bigr)=0,\qquad
\Re\tr\bigl(-i(K_{-j}-K_j)U\bigr)=0
\]
for all $U\in\mathfrak{sl}(2,\R)$. By nondegeneracy of the trace pairing on $\mathfrak{sl}(2,\R)$,
\begin{equation}\label{eq212frrf}
K_{-j}=-\overline{K_j}, \quad j\in\Z.
\end{equation}
Estimating $K_{-j}$ at height $y=\eta$ for $j\ge0$ and using \eqref{eq212frrf}, we obtain, for every $j\in\Z$,
\[
\|K_j\|=\|K_{-j}\|\le\sup_{x\in\T}\|K(x+i\eta)\|\,e^{-2\pi|j|\eta}.
\]
The Fourier series $\sum_{j\in\Z}K_j e^{2\pi ijz}$ converges absolutely on the two-sided strip $|\Im z|<\eta$, extending $K$ holomorphically across $\T$ with $K(x)=-\overline{K(x)}$ for $x\in\T$.

By analytic continuation, $H(z)=-iK(z)$ is real analytic on $\T$ and takes values in $\mathfrak{sl}(2,\R)$ on $\T$, with $H^2=-\frac14 I$ and $H\circ T=AHA^{-1}$. Thus $-2HJ$ is symmetric on $\T$ with $\det(-2HJ)=1$, and hence is everywhere definite on $\T$. Choosing $\epsilon\in\{1,-1\}$ so that $M=-2\epsilon HJ>0$ on $\T$, the identity $A^{-1}J=JA^T$ implies $M\circ T=AMA^T$. Setting $B=M^{-1/2}\in C^\omega(\T,\SL(2,\R))$ and $C=(B\circ T)AB^{-1}$, we obtain $CC^T=I$, as required.
\end{proof}

\begin{proposition}\label{prop25zrr}
Suppose $d=\deg A\ne0$, $L(\alpha,A)=0$, and $\omega(\alpha,A)=|d|$. Then $(\alpha,A)$ is analytically conjugate to an $\SO(2,\R)$-valued cocycle, and in particular is premonotonic with monotonicity sign $\operatorname{sgn}(d)$.
\end{proposition}
\begin{proof}
By Proposition~\ref{prop23allc}, $L(\alpha,\widetilde A)=L(\alpha,\widetilde A_\eta)-2\pi|d|\eta$ is real analytic on a neighborhood of $A$ in $\mathcal B_h$. Since $L(\alpha,\widetilde A)\ge0$ everywhere and $L(\alpha,A)=0$, the differential of $\widetilde A\mapsto L(\alpha,\widetilde A_\eta)$ vanishes at $A$. Proposition~\ref{prop204staa} therefore provides $B\in C^\omega(\T,\SL(2,\R))$ such that
\[
B(x+\alpha)A(x)B(x)^{-1}=R_{\phi(x)},\qquad \phi(x)=dx+\psi(x),
\]
for a 1-periodic real-analytic function $\psi:\T\to\R$. The $N$-step product of $R_\phi$ is a rotation with parameter
\[
Ndx+\frac{d\alpha N(N-1)}2+\sum_{k=0}^{N-1}\psi(x+k\alpha),
\]
whose derivative is $N\bigl(d+\frac1N\sum_{k=0}^{N-1}\psi'(x+k\alpha)\bigr)$. Since $\frac1N\sum_{k=0}^{N-1}\psi'\circ T^k\to\int_\T\psi'\,dx=0$ uniformly on $\T$ by unique ergodicity, this derivative has the strict sign of $d$ everywhere on $\T$ for all sufficiently large $N$.
\end{proof}

\section{Signedness and invariant matrix measures}\label{sec301cnn}

Let $\TT_{\alpha,A}H=A^T(H\circ T_\alpha)A$ be the pullback operator on symmetric matrix fields, abbreviated to $\TT$ when $(\alpha,A)$ is fixed. For a finite positive $\Sym_2\R$-valued Borel measure $\mu=(\mu_{ij})$ on $\T$, its \emph{trace measure} $\nu=\tr\mu$ is the positive scalar measure defined by
\[
\nu(E)=\tr\bigl(\mu(E)\bigr)=\mu_{11}(E)+\mu_{22}(E)
\qquad(E\subset\T\text{ Borel}).
\]
Every component of $\mu$ is absolutely continuous with respect to $\nu$, so the entrywise Radon--Nikodym theorem gives $d\mu=P\,d\nu$, where $P\ge0$ and $\tr P=1$ $\nu$-almost everywhere. The pairing of $\mu$ with $K\in C(\T,\Sym_2\R)$ is
\[
\int_\T\tr(K\,d\mu)
=\sum_{i,j=1}^2\int_\T K_{ij}\,d\mu_{ji}
=\int_\T\tr(KP)\,d\nu.
\]

\subsection{Hahn--Banach separation}

\begin{lemma}\label{lem31hbb}
Let $A\in C(\T,\SL(2,\R))$, $\alpha\in\R$, and $F\in C(\T,\Sym_2\R)$. Exactly one of the following alternatives holds:
\begin{enumerate}
\item There exists $H\in C(\T,\Sym_2\R)$ (which may be chosen to be a trigonometric polynomial) such that $F+\TT H-H>0$ everywhere on $\T$.
\item There exists a nonzero finite positive matrix-valued Borel measure $\mu$ on $\T$ such that
\begin{equation}\label{eq31mssr}
\mu(TE)=\int_E A(x)\,d\mu(x)\,A(x)^T\quad(\text{for all Borel }E\subset\T),
\qquad \int_\T\tr(F\,d\mu)\le0.
\end{equation}
\end{enumerate}
\end{lemma}
\begin{proof}
The two alternatives are mutually exclusive since integrating $F+\TT H-H>0$ against a measure satisfying \eqref{eq31mssr} yields $0<\int_\T\tr((F+\TT H-H)\,d\mu)=\int_\T\tr(F\,d\mu)\le0$. Conversely, if (1) fails, the affine subspace $F+\operatorname{ran}(\TT-I)$ is disjoint from the open cone of everywhere positive-definite fields in $C(\T,\Sym_2\R)$. By the Hahn--Banach separation theorem and the Riesz representation theorem \cite[Chapters~2, 5, and~6]{Rud87}, there is a nonzero continuous linear functional $\ell$ on $C(\T,\Sym_2\R)$ that is nonnegative on the positive cone, vanishes on $\operatorname{ran}(\TT-I)$, and satisfies $\ell(F)\le0$. By self-duality of the positive semidefinite cone, $\ell$ is represented by a nonzero finite positive matrix measure $\mu$. The vanishing of $\ell$ on $\operatorname{ran}(\TT-I)$ means that
\[
\int_\T\tr\bigl(H(Tx)A(x)\,d\mu(x)\,A(x)^T\bigr)
=\int_\T\tr\bigl(H(x)\,d\mu(x)\bigr) \quad \text{for all } H\in C(\T,\Sym_2\R),
\]
which is equivalent to the covariance in \eqref{eq31mssr}.

Finally, if $H\in C(\T,\Sym_2\R)$ satisfies $F+\TT H-H>0$, uniform approximation by symmetric trigonometric polynomials $\widetilde H$ preserves positivity because
\[
\|\TT(H-\widetilde H)-(H-\widetilde H)\|_\infty
\le\bigl(\|A\|_\infty^2+1\bigr)\|H-\widetilde H\|_\infty.
\]
\end{proof}

\subsection{The positive-exponent case}

For the proof of Proposition~\ref{prop32cnnc}, fix $\alpha\in\R\setminus\mathbb Q$ and $A\in C^\omega(\T,\SL(2,\R))$ satisfying $d=\deg A>0$, $L(\alpha,A)>0$, and $\omega(\alpha,A)=d$, and set $\FF(H)=Q_A+\TT H$.

The main step in the positive-exponent case is to construct a positive connection. The monodromy and Blaschke constructions in Sections~\ref{sec4mnnd}--\ref{sec6fnn} then turn this infinitesimal positivity into premonotonicity on the original circle.

\begin{proposition}\label{prop32cnnc}
Under the above hypotheses, there exists a real-analytic 1-periodic map $H:\T\to\Sym_2\R$ such that
\begin{equation}\label{eq303pst}
\FF(H)-H>0\qquad \text{on }\T.
\end{equation}
\end{proposition}

We begin by analyzing the structure of invariant positive matrix measures when $L(\alpha,A)>0$.

\begin{lemma}\label{lem302rkk}
Let $\alpha\in\R\setminus\mathbb Q$ and $A\in C^\omega(\T,\SL(2,\R))$ with $L(\alpha,A)>0$. Every finite positive matrix measure $\mu$ satisfying the covariance in \eqref{eq31mssr} is singular with respect to Lebesgue measure and has a rank-one Radon--Nikodym density with respect to its trace measure $\nu=\tr\mu$.
\end{lemma}
\begin{proof}
Since $T$ preserves Lebesgue measure, the Lebesgue decomposition of $\mu$ splits into separately covariant absolutely continuous and singular parts. If $M\,dx$ is the absolutely continuous part, then $M(Tx)=A(x)M(x)A(x)^T$ almost everywhere, so $\det M$ is $T$-invariant and hence constant almost everywhere by ergodicity. If $\det M>0$ almost everywhere, then $B=M^{-1/2}$ measurably conjugates $(\alpha,A)$ into $\SO(2,\R)$. By Poincar\'e recurrence to a set of positive measure where $B$ and $B^{-1}$ are bounded, $\|A_n(x)\|$ is bounded along return subsequences for almost every $x$ in that set, contradicting the almost-everywhere limit $n^{-1}\log\|A_n(x)\|\to L(\alpha,A)>0$.

Similarly, if $\det M=0$ and $M$ is not almost everywhere zero, then $\{x:M(x)\ne0\}$ has full measure and $M(x)=v(x)v(x)^T$ for a measurable vector field $v\ne0$ satisfying $A(x)v(x)=\pm v(Tx)$. Recurrence to a set of positive measure where $\|v\|$ is bounded away from $0$ and $\infty$ gives, for almost every $x$ in that set, a return subsequence $n_k\to\infty$ along which $\|A_{n_k}(x)v(x)\|$ remains bounded away from $0$ and $\infty$, contradicting Oseledets' theorem \cite[Chapter~4]{Via14} since the Lyapunov exponents are $\pm L(\alpha,A)\ne0$. Thus $M=0$ almost everywhere, and $\mu$ is singular.

For the density $P=d\mu/d\nu$, set $r_\mu(x)=\tr(A(x)P(x)A(x)^T)>0$ for $\nu$-almost every $x$. The covariance~\eqref{eq31mssr} reads
\begin{equation}\label{eq34dnns}
\nu(TE)=\int_E r_\mu\,d\nu,\qquad
P(Tx)=\frac{A(x)P(x)A(x)^T}{r_\mu(x)}\quad(\nu\text{-a.e.}).
\end{equation}
Taking determinants gives $\sqrt{\det P(Tx)}\,r_\mu(x)=\sqrt{\det P(x)}$ $\nu$-almost everywhere, which implies that $\sqrt{\det P}\,d\nu$ is a finite $T$-invariant Borel measure on $\T$. By unique ergodicity of $T$, $\sqrt{\det P}\,d\nu$ must be a multiple of Lebesgue measure; since $\nu$ is singular, $\det P=0$ $\nu$-almost everywhere. As $\tr P=1$, $P$ has rank one $\nu$-almost everywhere.
\end{proof}

\subsection{A boundary resolvent estimate}

All Hilbert spaces below are complex, with inner products conjugate-linear in the first argument.

\begin{lemma}\label{lem33rsvv}
Let $\eta>0$ and let $U(z)$ be a 1-periodic holomorphic family of bounded operators on a Hilbert space defined on a neighborhood of the strip $0\le\Im z\le\eta$, such that $U(x)$ is unitary for all $x\in\R$. Fix $r,C>0$, and suppose that $U(z)-(1+r)I$ is invertible throughout $0\le\Im z\le\eta$ with
\[
\sup_{x\in\T}\|(U(x+i\eta)-(1+r)I)^{-1}\|\le C,\qquad Cr<1.
\]
Then for every $x_0\in\T$ and every unit vector $\phi$ satisfying $U(x_0)\phi=\phi$,
\[
\langle\phi,-iU(x_0)^*U'(x_0)\phi\rangle
\ge\frac r\eta\log\frac1{Cr}>0.
\]
\end{lemma}
\begin{proof}
By continuity and 1-periodicity, $(U(z)-(1+r)I)^{-1}$ is uniformly bounded on the closed strip $0\le\Im z\le\eta$. Since $U(x)$ is unitary for $x\in\R$, we have $\|(U(x)-(1+r)I)^{-1}\|\le 1/r$ on $\Im z=0$. Applying Hadamard's three-lines theorem \cite[Theorem~12.8]{Rud87} to the bounded scalar functions $z\mapsto\langle\psi_1,(U(z)-(1+r)I)^{-1}\psi_2\rangle$, holomorphic on a neighborhood of $0\le\Im z\le\eta$, and taking the supremum over unit vectors $\psi_1,\psi_2$ gives
\[
\|(U(x+iy)-(1+r)I)^{-1}\|\le r^{-1+y/\eta}C^{y/\eta}, \quad 0\le y\le\eta.
\]
For a unit vector $\phi$ with $U(x_0)\phi=\phi$, this gives
\begin{equation}\label{eq306vcbb}
\|(U(x_0+iy)-(1+r)I)\phi\|\ge r^{1-y/\eta}C^{-y/\eta}, \quad 0\le y\le\eta,
\end{equation}
where both sides equal $r>0$ at $y=0$. Since $\partial_y U(x_0)=iU'(x_0)$, $U(x_0)\phi=\phi$, and $-iU(x_0)^*U'(x_0)$ is self-adjoint by unitarity of $U(x)$ on $\R$, differentiating the norm of $(U(x_0+iy)-(1+r)I)\phi=-r\phi+iy\,U'(x_0)\phi+O(y^2)$ at $y=0^+$ in \eqref{eq306vcbb} yields
\[
\langle\phi,-iU(x_0)^*U'(x_0)\phi\rangle
=\left.\frac{d}{dy}\right|_{y=0}\|(U(x_0+iy)-(1+r)I)\phi\|
\ge\frac r\eta\log\frac1{Cr}>0.
\]
\end{proof}

\subsection{Scattering representation and proof of Proposition~\ref{prop32cnnc}}

Using the transfer-to-scattering conversion of \cite[Proposition~2]{MSB13}, with our ordering of the channels, we associate to the Cayley matrix $\widehat A=\left(\begin{smallmatrix}a&b\\c&e\end{smallmatrix}\right)$ in \eqref{eq202cyy} the scattering matrix
\begin{equation}\label{eq37sctt}
S(z)=\frac1{a(z)}\begin{pmatrix}1&-b(z)\\c(z)&1\end{pmatrix}.
\end{equation}
On a strip $|\Im z|<h$ where $a$ is nonvanishing, $S(z)$ is holomorphic, and for $x\in\T$ the $\SU(1,1)$ relations imply that $S(x)$ is unitary. Moreover, for $\xi=(u,w)^T$ and $\xi_+=(u_+,w_+)^T$, the transfer relation $\xi_+=\widehat A\xi$ is equivalent to the scattering relation
\begin{equation}\label{eq308lcsc}
\begin{pmatrix}u\\w_+\end{pmatrix}=S\begin{pmatrix}u_+\\w\end{pmatrix}.
\end{equation}
We also introduce the complex spectral deformation
\[
A^{(t)}(z)=\exp(tJ/2)A(z)\exp(tJ/2),\qquad t\in\C,
\]
whose Cayley matrix has entries $e^{-it}a,b,c,e^{it}e$ and whose associated scattering matrix is simply
\begin{equation}\label{eq39stdf}
S^{(t)}(z)=e^{it}S(z).
\end{equation}

\begin{lemma}\label{lem344rcc}
There exist $\eta>0$ and $s_0>0$ such that $(\alpha,A^{(is)}_y)$ is uniformly hyperbolic for all $0<s\le s_0$ and $0\le y\le\eta$. Moreover, at $y=\eta$, the exponential dichotomy constants of $(\alpha,A^{(is)}_\eta)$ are uniform in $0\le s\le s_0$.
\end{lemma}
\begin{proof}
Fix $\eta\in(0,h)$ in the affine interval~\eqref{eq208aff} small enough that $a$ has no zeros on $|\Im z|\le\eta$. By openness of uniform hyperbolicity, there exists $s_0>0$ such that $(\alpha,A^{(is)}_\eta)$ remains uniformly hyperbolic with uniform dichotomy constants and acceleration $d$ for all $s\in[0,s_0]$.

For each $s\in(0,s_0]$ and $x\in\T$, the Cayley matrix $\widehat A^{(is)}(x)$ maps $\overline\D$ strictly into $\D$, because $\exp(isJ/2)$ acts projectively as multiplication by $e^{-s}<1$ while $\widehat A(x)$ preserves $\D$. Thus $A^{(is)}$ is uniformly hyperbolic at $y=0$, with unstable section $(1,m)^T$ ($|m|<1$) and unstable multiplier $\lambda=e^sa+bm$. Since $|e^sa|>|b|\ge|bm|$ on $\T$, the homotopy $e^sa+\tau bm$ ($0\le\tau\le1$) is nonvanishing on $\T$. Thus $\wind\lambda=\wind a=-d$ and $\omega(\alpha,A^{(is)})=d$. Because $y\mapsto L(\alpha,A^{(is)}_y)$ is convex and has slope $2\pi d$ at both $y=0$ and $y=\eta$, it is affine with positive values throughout $[0,\eta]$. In particular, at every interior height $y_0\in(0,\eta)$, $t\mapsto L(\alpha,A^{(is)}_{y_0+t})$ is affine in a two-sided neighborhood of $t=0$ with $L(\alpha,A^{(is)}_{y_0})>0$, so Avila's regularity theorem \cite[Theorem~6]{Av15} ensures that $(\alpha,A^{(is)}_{y_0})$ is uniformly hyperbolic as well.
\end{proof}

The next lemma derives a resolvent bound from the exponential dichotomy, uniformly over covariant representations.

\begin{lemma}\label{lem35repp}
Let $W$ be a unitary operator on a Hilbert space $\mathcal K$ equipped with a unital $*$-representation $f\mapsto M_f$ of $C(\T)$ satisfying $WM_fW^{-1}=M_{f\circ T}$. We use the same notation for the componentwise extension of $W$ and the entrywise extension of the representation to matrix-valued functions. If $D\in C(\T,\SL(2,\C))$ is uniformly hyperbolic, then $W-M_D$ is invertible on $\mathcal K\otimes\C^2$, with $\|(W-M_D)^{-1}\|$ bounded solely in terms of the continuous invariant frame and exponential dichotomy constants of $(\alpha,D)$, independently of $(\mathcal K,W,M)$.
\end{lemma}
\begin{proof}
Since complex line bundles over $\T$ are trivial, the invariant splitting admits a continuous frame $V\in C(\T,\SL(2,\C))$ diagonalizing $D$: $D(x)V(x)=V(Tx)\diag(\lambda(x),\lambda(x)^{-1})$. The unstable multipliers $\lambda_n(x)=\prod_{j=0}^{n-1}\lambda(T^j x)$ satisfy $|\lambda_n(x)|\ge ce^{\gamma n}$ for $n\ge1$. Then
\[
(W-M_D)M_V=M_{V\circ T}\bigl(W-M_{\diag(\lambda,\lambda^{-1})}\bigr).
\]
On the expanding and contracting components, we invert via the convergent Neumann series
\[
\begin{gathered}
(W-M_\lambda)^{-1}=-\sum_{n=0}^\infty (M_{\lambda^{-1}}W)^n M_{\lambda^{-1}},\\
(W-M_{\lambda^{-1}})^{-1}=\sum_{n=0}^\infty (W^{-1}M_{\lambda^{-1}})^n W^{-1},
\end{gathered}
\]
since $\|(M_{\lambda^{-1}}W)^n\|\le\sup_x|\lambda_n(x)|^{-1}\le c^{-1}e^{-\gamma n}$ for $n\ge1$, and similarly for $(W^{-1}M_{\lambda^{-1}})^n$.
\end{proof}

\begin{proof}[Proof of Proposition~\ref{prop32cnnc}]
By Lemma~\ref{lem31hbb}, it suffices to prove positivity of the signedness pairing for every invariant positive matrix measure $\mu$ normalized by $\tr\mu(\T)=1$. Its trace measure $\nu$ is a probability measure, and Lemma~\ref{lem302rkk} gives $P=vv^T$ $\nu$-almost everywhere for a measurable unit vector field $v:\T\to\R^2$. The weight in~\eqref{eq34dnns} therefore satisfies $r_\mu(x)=\|A(x)v(x)\|^2$, with $0<\|A^{-1}\|_\infty^{-2}\le r_\mu(x)\le\|A\|_\infty^2$ $\nu$-almost everywhere. By \eqref{eq34dnns}, $\nu$ is quasi-invariant under $T$ and $v(Tx)=\sigma(x)r_\mu(x)^{-1/2}A(x)v(x)$ for a measurable sign $\sigma:\T\to\{1,-1\}$ and $\nu$-almost every $x$. On $\mathcal K=L^2(\nu)$, the operator
\[
(Wf)(x)=\sigma(x)\sqrt{r_\mu(x)}\,f(Tx),\qquad
(W^{-1}g)(x)=\frac{\sigma(T^{-1}x)}{\sqrt{r_\mu(T^{-1}x)}}\,g(T^{-1}x),
\]
is unitary by \eqref{eq34dnns}, satisfies $WM_fW^{-1}=M_{f\circ T}$, and obeys $Wv=Av$ componentwise. Here $M_f$ acts by multiplication by $f$ on $\mathcal K$, and $M_D$ acts by pointwise matrix multiplication on $\mathcal K\otimes\C^2$. On $\mathcal K\otimes\C^2$, define the unitary operators $P_{\rm in}(u,w)=(Wu,w)$ and $P_{\rm out}(u,w)=(u,Ww)$, and set
\[
U_\mu(z)=P_{\rm out}^{-1}M_{S(\cdot+z)}P_{\rm in}.
\]
Because $W$ and $\nu$ are independent of $z$, $z\mapsto U_\mu(z)$ is holomorphic in the operator norm on $|\Im z|<h$ and unitary for $z\in\R$. Moreover, for the unit vector $\xi=\Q v\in\mathcal K\otimes\C^2$, the transfer equation $W\xi=\widehat A\xi$ is equivalent by \eqref{eq308lcsc} to $U_\mu(0)\xi=\xi$.

We claim that the boundary phase derivative at $\xi$ is given by
\begin{equation}\label{eq311actn}
\langle\xi,-iU_\mu(0)^*U_\mu'(0)\xi\rangle=\int_\T\tr(Q_A\,d\mu).
\end{equation}
Indeed, we hold $\xi=(u,w)^T$ fixed at each point and differentiate only the matrix coefficients in $\binom{u}{w_+}=S(x)\binom{u_+}{w}$ with respect to $x\in\T$, with $\xi_+=\widehat A(x)\xi=(u_+,w_+)^T$. Unitarity of $S(x)$ gives
\[
\begin{aligned}
\begin{pmatrix}u_+\\w\end{pmatrix}^* S(x)^*S'(x)\begin{pmatrix}u_+\\w\end{pmatrix}
&=\bar w_+w_+'-\bar u_+u_+'\\
&=-\xi_+^*\diag(1,-1)\widehat A'(x)\xi\\
&=i\,v(x)^TQ_A(x)v(x),
\end{aligned}
\]
where the last step uses $\Q^*(i\diag(1,-1))\Q=-J$. Multiplying by $-i$ and integrating against $d\nu$ yields \eqref{eq311actn}.

We next transfer the dichotomy bound to the resolvent of the scattering family. For $0\le\Im z\le\eta$ and $0<s\le s_0$ as in Lemma~\ref{lem344rcc}, a direct matrix verification from \eqref{eq37sctt}, \eqref{eq39stdf}, and $ae-bc=1$ yields the operator factorization
\begin{equation}\label{eq312fctt}
W-M_{\widehat A^{(is)}(\cdot+z)}
=-e^{-s}M_{\left(\begin{smallmatrix}-e^sa&0\\-c&1\end{smallmatrix}\right)(\cdot+z)}P_{\rm out}\bigl(U_\mu(z)-e^sI\bigr).
\end{equation}
Set
\[
K_s(z)=\begin{pmatrix}-e^sa(z)&0\\-c(z)&1\end{pmatrix},\qquad
\mathcal D_{s,z}=W-M_{\widehat A^{(is)}(\cdot+z)}.
\]
Since $a$ is nonvanishing on $0\le\Im z\le\eta$, the matrix multiplier $M_{K_s(\cdot+z)}$ is invertible with uniformly bounded norm and inverse norm. Lemmas~\ref{lem344rcc} and~\ref{lem35repp} show that $\mathcal D_{s,z}$ is invertible throughout this strip for $0<s\le s_0$. Inverting \eqref{eq312fctt} gives
\[
(U_\mu(z)-e^sI)^{-1}
=-e^{-s}\mathcal D_{s,z}^{-1}M_{K_s(\cdot+z)}P_{\rm out},
\]
and hence
\[
\|(U_\mu(z)-e^sI)^{-1}\|
\le e^{-s}\|\mathcal D_{s,z}^{-1}\|
\sup_{t\in\T}\|K_s(t+z)\|.
\]
At the upper boundary, the factorization and inverse formula also hold for $s=0$. The cocycles $A^{(is)}(\cdot+\theta+i\eta)$, parametrized by $(s,\theta)\in[0,s_0]\times\T$, form a compact family of uniformly hyperbolic cocycles. Local choices of invariant frames and a finite subcover give common frame bounds and dichotomy constants. Lemma~\ref{lem35repp}, the uniform bound for $K_s$, and the unitarity of $W$ and $P_{\rm out}$ therefore give a constant $C>0$, independent of $\mu$ and $s$, such that
\[
\sup_{x\in\T}\|(U_\mu(x+i\eta)-e^sI)^{-1}\|\le C, \quad 0\le s\le s_0.
\]
Choosing $s\in(0,s_0]$ small enough that $r=e^s-1$ satisfies $Cr<1$, Lemma~\ref{lem33rsvv} and~\eqref{eq311actn} give
\[
\int_\T\tr(Q_A\,d\mu)\ge\frac r\eta\log\frac1{Cr}>0
\]
for every normalized invariant positive matrix measure $\mu$. Applying Lemma~\ref{lem31hbb} with $F=Q_A$ then yields a real-analytic $H$ satisfying \eqref{eq303pst}.
\end{proof}

\section{Quarter-turn monodromy}\label{sec4mnnd}

Integrating the positive connection of Proposition~\ref{prop32cnnc} gives a monotonic cocycle on the universal cover $\R$. We choose the connection so that the gauge has quarter-turn monodromy, which will be removed in Section~\ref{sec6fnn}.

\begin{lemma}\label{lem41gga}
Let $B:\R\to\SL(2,\R)$ be differentiable and set
\[
H=B^T(-J)B',\qquad C(x)=B(Tx)A(x)B(x)^{-1}.
\]
Then $H$ is symmetric and
\begin{equation}\label{eq41ggid}
B(x)^TQ_C(x)B(x)=Q_A(x)+A(x)^TH(Tx)A(x)-H(x).
\end{equation}
Conversely, given any 1-periodic $H\in C^\omega(\T,\Sym_2\R)$ and $B(0)\in\SL(2,\R)$, the solution $B:\R\to\SL(2,\R)$ of $B'=BJH$ is real analytic and satisfies \eqref{eq41ggid} and
\[
B(x+1)=MB(x), \quad x\in\R,
\]
for a constant monodromy matrix $M\in\SL(2,\R)$.
\end{lemma}
\begin{proof}
Differentiating $C=(B\circ T)AB^{-1}$ and using $D^TJD=J$ for $D\in\SL(2,\R)$ gives \eqref{eq41ggid}, while differentiating $B^TJB=J$ shows that $H=B^T(-J)B'$ is symmetric. Conversely, since $\tr(JH)=0$, the linear differential equation $B'=BJH$ preserves $\SL(2,\R)$ and recovers $B^T(-J)B'=H$. When $H$ is 1-periodic, $B(x+1)$ and $B(x)$ solve the same differential equation, so $B(x+1)=MB(x)$ with $M=B(1)B(0)^{-1}\in\SL(2,\R)$.
\end{proof}

Recall \cite[Section~I.1]{dMvS93} that every element of the universal cover $\widetilde{\SL}(2,\R)$ acts on the lifted projective circle $\R$ with a well-defined translation number. We normalize this coordinate so that a full rotation by $2\pi$ shifts it by $1$; the projective circle then has period $1/2$. The translation number is $\rho(\widetilde g)=\lim_{n\to\infty}(\widetilde g^n(t)-t)/n$, which is continuous, conjugacy-invariant, and satisfies $\rho(z^k\widetilde g)=\rho(\widetilde g)+k$ for the generator $z$ of the deck group of $\widetilde{\SL}(2,\R)\to\SL(2,\R)$ corresponding to a full turn; this element is central. Hyperbolic and parabolic elements of $\widetilde{\SL}(2,\R)$ have translation numbers in $\frac12\Z$. For an elliptic matrix conjugate in $\SL(2,\R)$ to $R_\theta$, its lifted translation number is congruent to $\theta$ modulo $\Z$. Thus any element with $\rho(\widetilde g)\in\Z+\frac34$ projects to an elliptic matrix conjugate to the quarter-turn rotation $R_{-1/4}$.

\begin{lemma}\label{lem402qtt}
Suppose $d=\deg A>0$ and $H\in C^\omega(\T,\Sym_2\R)$ satisfies~\eqref{eq303pst}. Then there exists a real-analytic map $B_0:\R\to\SL(2,\R)$ such that
\begin{equation}\label{eq403qrt}
B_0(x+1)=R_{-1/4}B_0(x), \quad x\in\R,
\end{equation}
and $C=(B_0\circ T)AB_0^{-1}$ is strictly positively monotonic. In particular, both $B_0$ and $C$ are 4-periodic.
\end{lemma}
\begin{proof}
Set $H_0=H$ and $H_1=\FF(H_0)=Q_A+\TT H_0$. Since $\FF(H_1)-H_1=\TT(\FF(H_0)-H_0)>0$, every convex combination $H_s=(1-s)H_0+sH_1$ ($s\in[0,1]$) also satisfies $\FF(H_s)-H_s>0$. Let $\Phi_s:\R\to\SL(2,\R)$ be the solution of $\Phi_s'=\Phi_sJH_s$ with $\Phi_s(0)=I$, and let $\rho(s)$ be the translation number of the lifted monodromy $\widetilde\Phi_s(1)\in\widetilde{\SL}(2,\R)$.

The endpoint translation numbers satisfy
\begin{equation}\label{eq44shdd}
\rho(1)=\rho(0)+d.
\end{equation}
Indeed, by Lemma~\ref{lem41gga}, the path $\Psi(x)=\Phi_0(x+\alpha)A(x)$ satisfies $\Psi^T(-J)\Psi'=Q_A+\TT H_0=H_1$, so $\Phi_1(x)=\Psi(0)^{-1}\Psi(x)$. Lifting to $\widetilde{\SL}(2,\R)$ with $\widetilde\Phi_s(0)=e$, the loop $A$ satisfies $\widetilde A(x+1)=z^d\widetilde A(x)$ and $\widetilde\Phi_0(x+1)=\widetilde\Phi_0(1)\widetilde\Phi_0(x)$, where $z$ is the deck-group generator fixed above. Thus $\widetilde\Psi(x+1)=z^d\widetilde\Phi_0(1)\widetilde\Psi(x)$, and hence $\widetilde\Phi_1(1)=\widetilde\Psi(0)^{-1}z^d\widetilde\Phi_0(1)\widetilde\Psi(0)$. Conjugacy invariance and centrality of $z^d$ give~\eqref{eq44shdd}.

By continuity and~\eqref{eq44shdd}, the image $\rho([0,1])$ contains $[\rho(0),\rho(0)+d]$. Since $d\ge1$, this interval contains $k+\frac34$ for some $k\in\Z$. Fix $s\in[0,1]$ with $\rho(s)=k+\frac34$. Then the monodromy $\Phi_s(1)$ is conjugate in $\SL(2,\R)$ to $R_{-1/4}$: choosing $G\in\SL(2,\R)$ with $G\Phi_s(1)G^{-1}=R_{-1/4}$ and setting $B_0(x)=G\Phi_s(x)$, we obtain $B_0(x+1)=R_{-1/4}B_0(x)$ and $B_0^T(-J)B_0'=H_s$. Lemma~\ref{lem41gga} then ensures that $C=(B_0\circ T)AB_0^{-1}$ is strictly positively monotonic, and $R_{-1/4}^4=I$ gives 4-periodicity.
\end{proof}

\section{A Blaschke majorant near the boundary}

To untwist the quarter-turn monodromy of $B_0$, we establish a strip majorization lemma for bounded holomorphic functions with unimodular boundary values (see \cite[Chapters~I--II]{Gar07} for background on boundary behavior of bounded analytic functions).

\begin{lemma}\label{lem51mjj}
Let $f$ be a 1-periodic holomorphic function on the strip $0<\Im z<h$ with $|f|<1$. Suppose that its nontangential boundary values satisfy $|f(x)|=1$ for almost every $x\in\T$, and that $f$ does not extend holomorphically across the entire real circle $\T$. Then there exist $x_0\in\T$ and $\delta>0$ such that the single Blaschke factor
\begin{equation}\label{eq51blsk}
\beta(z)=\frac{e^{2\pi iz}-\mathfrak a}{1-\overline{\mathfrak a}e^{2\pi iz}},\qquad
\mathfrak a=e^{2\pi i x_0-2\pi\delta},
\end{equation}
satisfies
\begin{equation}\label{eq502mjj}
|f(x+iy)|<|\beta(x+iy)|<1
\qquad\text{for all }x\in\T,\ 0<y\le\delta/2.
\end{equation}
Moreover, $\beta$ is holomorphic and nonvanishing on a neighborhood of the closed strip $|\Im z|\le\delta/2$, with $|\beta(x)|=1$ and $\wind\beta=1$ on $\T$.
\end{lemma}
\begin{proof}
Set $u=-\log|f|>0$ and $g_\beta=-\log|\beta|$. Fix $h_0\in(0,\min\{h,1/2\})$ such that $f$ has no zeros on the horizontal circle $\Im z=h_0$, so that $\min_{x\in\T}u(x+ih_0)>0$. If zeros of $f$ approach the real axis, we place the zero of the Blaschke factor at one of them. Otherwise, the singular boundary measure in the Poisson--Herglotz representation supplies the required comparison.

\medskip
\noindent{\it Case 1: Zeros of $f$ accumulate at $\T$.}
Choose a zero $z_0=x_0+i\delta$ of $f$ with $\delta\in(0,h_0)$ small enough that $\max_{x\in\T}g_\beta(x+ih_0)<\min_{x\in\T}u(x+ih_0)$. The quotient $g=f/\beta$ extends holomorphically across $z_0$, since $f$ vanishes there at least to the order of $\beta$. It is bounded on $\T\times(0,h_0)$, satisfies $|g|<1$ on $\Im z=h_0$, and obeys
\[
\limsup_{y\downarrow0}\sup_{x\in\T}|g(x+iy)|\le1,
\]
because $|f|<1$ and $|\beta(x+iy)|\to1$ uniformly as $y\downarrow0$. Applying the maximum modulus principle on $\T\times[\tau,h_0]$ and letting $\tau\downarrow0$ gives $|g|\le1$ throughout the strip; the strict upper-boundary inequality and the strong maximum principle give $|g|<1$ in its interior. Since $\beta$ has no zeros on $0<\Im z\le\delta/2$, we obtain $|f|<|\beta|$ there, proving \eqref{eq502mjj}.

\medskip
\noindent{\it Case 2: $f$ has no zeros in $0<\Im z\le h_0$ (after shrinking $h_0$ if needed).}
In this case $u$ is positive and harmonic on $0<\Im z\le h_0$. Since $\int_\T u(x+i\tau)\,dx$ is affine in $\tau\in(0,h_0]$ and bounded below by $0$, the family of measures $u(x+i\tau)\,dx$ has uniformly bounded total mass as $\tau\downarrow0$. Passing to a weak-$*$ convergent subsequence of the lower boundary measures in the Dirichlet representation on $\tau\le\Im z\le h_0$ as $\tau\downarrow0$ yields the Poisson--Herglotz representation
\begin{equation}\label{eq53pssn}
u(x+iy)=\int_\T P_{h_0}(x-t,y)\,d\mu(t)+v(x,y), \quad 0<y\le h_0,
\end{equation}
where $\mu$ is a finite positive Borel measure on $\T$, $v\ge0$ is the harmonic extension of the upper boundary data $u(\cdot+ih_0)$ with $v(\cdot,0)=0$, and $P_{h_0}(t,y)$ is the periodized lower-boundary strip Poisson kernel. Since $\lim_{y\downarrow0}u(x+iy)=0$ for Lebesgue-almost every $x\in\T$ by hypothesis, Fatou's theorem implies that $\mu$ is singular with respect to Lebesgue measure. Moreover, $\mu\ne0$: otherwise $u=v$ would extend harmonically across $\T$ by odd reflection, and hence $f$ would extend holomorphically across $\T$ by local Schwarz reflection, contradicting the hypothesis.

The lower-boundary Poisson kernel is obtained by periodizing the positive strip kernel
\[
\frac1{2h_0}\,
\frac{\sin(\pi y/h_0)}
{\cosh(\pi t/h_0)-\cos(\pi y/h_0)}.
\]
Its Taylor expansion near $(t,y)=(0,0)$, together with positivity and compactness away from that point, gives the lower bound
\begin{equation}\label{eq504psbb}
P_{h_0}(t,y)\ge\frac{c\,y}{\dist_\T(t,0)^2+y^2} \quad (t\in\T,\ 0<y\le h_0/2)
\end{equation}
for a constant $c>0$ depending only on $h_0$, where $\dist_\T(s,t)=\min_{k\in\Z}|s-t-k|$ is the distance on $\T$. For any $x_0\in\T$ and $\delta\in(0,h_0/2)$, let $I_\delta(x_0)=\{t\in\T:\dist_\T(t,x_0)\le\delta\}$. For $t\in I_\delta(x_0)$ and $0<y\le\delta/2$, the triangle inequality gives $\dist_\T(x,t)^2+y^2\le 2\bigl(\dist_\T(x,x_0)^2+\delta^2\bigr)+\frac14\delta^2\le 3\bigl(\dist_\T(x,x_0)^2+\delta^2\bigr)$, so~\eqref{eq53pssn} and~\eqref{eq504psbb} imply
\[
u(x+iy)\ge\frac{c}{3}\,\mu(I_\delta(x_0))\frac{y}{\dist_\T(x,x_0)^2+\delta^2}, \quad 0<y\le\delta/2.
\]
On the other hand, a direct estimate from \eqref{eq51blsk} gives, for $0<y\le\delta/2$,
\[
g_\beta(x+iy)
=\frac12\log\frac{\cosh(2\pi(\delta+y))-\cos(2\pi(x-x_0))}
{\cosh(2\pi(\delta-y))-\cos(2\pi(x-x_0))}
\le\frac{C\,\delta y}{\dist_\T(x,x_0)^2+\delta^2},
\]
where $C>0$ is an absolute constant. Since $\mu$ is nonzero and singular, there exist $x_0\in\T$ and $\delta\in(0,h_0/2)$ with $\mu(I_\delta(x_0))>(3C/c)\delta$. Otherwise, covering a Lebesgue-null set carrying $\mu$ by arcs of radius less than $h_0/2$ and arbitrarily small total length would give $\mu(\T)=0$. For this pair $(x_0,\delta)$, the preceding estimates yield $u>g_\beta$ on $0<y\le\delta/2$, establishing \eqref{eq502mjj}. Finally, the zeros and poles of $\beta$ lie at heights $\delta$ and $-\delta$, respectively. Thus $\beta$ is holomorphic and nonvanishing near $|\Im z|\le\delta/2$, with $|\beta|=1$ and $\wind\beta=1$ on $\T$.
\end{proof}

\section{Descent to the original circle}\label{sec6fnn}

The following criterion will be applied after constructing a periodic gauge.

\begin{lemma}\label{lem601ohh}
Let $\eps>0$ and $\alpha\in\R\setminus\mathbb Q$, and let $D\in C^\omega(\T,\SL(2,\R))$ have degree $d>0$ and extend holomorphically to a neighborhood of the closed strip $0\le\Im z\le\eps$. Suppose that $D_\eps$ is uniformly hyperbolic with acceleration $d$, and that its Cayley unstable and stable sections take values in $\D$ and $\widehat\C\setminus\overline\D$, respectively. Then $D_N$ is strictly positively monotonic on $\T$ for all sufficiently large $N$.
\end{lemma}
\begin{proof}
We first exclude zeros of the upper-left entries of sufficiently long transfer products in Cayley coordinates on the closed strip. At height $\eps$, write the holomorphic unstable and stable directions of $\widehat D_\eps$ as $(1,p(x))^T$ and $(q(x),1)^T$, where $p,q:\T\to\D$ are analytic and satisfy $|p(x)|,|q(x)|\le\rho<1$ for all $x\in\T$. Setting $\Delta=1-pq$, the 1-periodic holomorphic frame
\[
V(x)=\begin{pmatrix}1&q(x)/\Delta(x)\\p(x)&1/\Delta(x)\end{pmatrix}\in\SL(2,\C)
\]
diagonalizes $\widehat D_\eps(x)$ with unstable multiplier $\lambda(x)$, where $\wind\lambda=-d$ by \eqref{eq206wdll} and $|\lambda_N(x)|\ge ce^{\gamma N}$ for $N\ge1$ with $c,\gamma>0$. Thus the upper-left entry of $\widehat D_N(x+i\eps)$ is given by
\begin{equation}\label{eq61dmm}
a_N(x+i\eps)
=\frac{\lambda_N(x)}{\Delta(x)}
-\frac{q(x+N\alpha)p(x)\lambda_N(x)^{-1}}{\Delta(x+N\alpha)}.
\end{equation}
Since $|pq|<1$, $\wind\Delta=0$, and for $N$ large the first term in \eqref{eq61dmm} dominates the second on $\T$, so $\wind a_N(\cdot+i\eps)=\wind\lambda_N=-Nd$. Because $\widehat D_N(x)\in\SU(1,1)$ on $\T$ has degree $Nd$, we also have $\wind a_N=-Nd$ on $\Im z=0$. By the argument principle on the cylinder $\T\times[0,\eps]$, $a_N(z)$ has no zeros in $0\le\Im z\le\eps$ for all sufficiently large $N$.

We next obtain uniform disk contraction at the upper boundary. Decomposing $(1,\zeta)^T$ ($\zeta\in\overline\D$) in the frame $V(x)$ gives
\[
\begin{pmatrix}1\\\zeta\end{pmatrix}
=\frac{1-q(x)\zeta}{\Delta(x)}\begin{pmatrix}1\\p(x)\end{pmatrix}
+\bigl(\zeta-p(x)\bigr)\begin{pmatrix}q(x)/\Delta(x)\\1/\Delta(x)\end{pmatrix}.
\]
Denote these two coefficients by $a_*=(1-q\zeta)/\Delta$ and $b_*=\zeta-p$. Since $|1-q\zeta|\ge1-\rho$ and $|\Delta|\le1+\rho^2$,
\[
\left|\frac{b_*}{a_*}\right|
\le K_\rho:=\frac{(1+\rho)(1+\rho^2)}{1-\rho}.
\]
After $N$ iterates, set $e_* =\lambda_N(x)^{-2}b_*/a_*$. The projective image is
\[
\widehat D_N(x+i\eps)\cdot\zeta
=\frac{p(x+N\alpha)+e_*/\Delta(x+N\alpha)}
{1+e_*q(x+N\alpha)/\Delta(x+N\alpha)},
\]
where $|e_*|\le K_\rho c^{-2}e^{-2\gamma N}$. Its difference from $p(x+N\alpha)$ equals
\[
\frac{e_*}{1+e_*q(x+N\alpha)/\Delta(x+N\alpha)}.
\]
Because $|\Delta|\ge1-\rho^2$, the denominator has modulus at least $1/2$ for all sufficiently large $N$, uniformly in $x$ and $\zeta\in\overline\D$. Thus
\[
\left|\widehat D_N(x+i\eps)\cdot\zeta-p(x+N\alpha)\right|
\le2K_\rho c^{-2}e^{-2\gamma N}.
\]
For all sufficiently large $N$, $\widehat D_N(x+i\eps)$ maps $\overline\D$ into $\{|w|\le(1+\rho)/2\}$. Since $\det\widehat D_N=1$ precludes cancellation of a pole, its M\"obius denominator has no zero in $\overline\D$, which gives $|b_N/a_N|<1$ on $\Im z=\eps$. Since $|b_N/a_N|<1$ on $\Im z=0$ and $a_N$ has no zeros in $0\le\Im z\le\eps$, the maximum principle gives $|b_N(z)/a_N(z)|<1$ throughout the closed strip.

The Cauchy--Riemann equations turn inward motion under complexification into positive angular motion on the real circle. For each $\zeta\in\partial\D$, the denominator $a_N(z)+b_N(z)\zeta$ is nonvanishing on $0\le\Im z\le\eps$, so $\Phi_\zeta(z)=\widehat D_N(z)\cdot\zeta$ is holomorphic on a neighborhood of $0\le\Im z\le\eps$, has modulus one on $\Im z=0$, and satisfies $|\Phi_\zeta(x+i\eps)|\le(1+\rho)/2<1$. Applying the maximum modulus principle to $\Phi_\zeta$ gives $|\Phi_\zeta|\le1$ throughout the closed strip. Hence $u_\zeta(z)=-\log|\Phi_\zeta(z)|$, interpreted as $+\infty$ at zeros of $\Phi_\zeta$, is superharmonic on $0<\Im z<\eps$. Since $|\Phi_\zeta|=1$ on the lower boundary, $\Phi_\zeta$ has no zeros near that boundary; thus $u_\zeta$ is harmonic there and vanishes on $\Im z=0$. On the upper boundary, $u_\zeta\ge\log\frac{2}{1+\rho}$. The minimum principle on $\T\times[0,\eps]$ gives
\[
u_\zeta(x+iy)\ge\frac{y}{\eps}\log\frac{2}{1+\rho}, \quad x\in\T,\ 0\le y\le\eps.
\]
Differentiating at $y=0^+$ and applying the Cauchy--Riemann equations gives
\[
\partial_x\arg\Phi_\zeta(x)
=-\left.\partial_y\log|\Phi_\zeta(x+iy)|\right|_{y=0}
=\partial_y u_\zeta(x)\ge\frac{1}{\eps}\log\frac{2}{1+\rho}>0
\]
for all $(x,\zeta)\in\T\times\partial\D$. For $v=(v_1,v_2)^T\in\R^2\setminus\{0\}$, take $\zeta=(v_1+iv_2)/(v_1-iv_2)$. Since the Cayley angle is twice the vector angle, \eqref{eq13angg} gives
\[
\frac{v^TQ_{D_N}(x)v}{\|D_N(x)v\|^2}
=\frac12\partial_x\arg\Phi_\zeta(x)
\ge\frac1{2\eps}\log\frac2{1+\rho}>0.
\]
Thus $Q_{D_N}>0$ on $\T$ for every sufficiently large $N$.
\end{proof}

We also record the elementary topological obstruction to continuous real invariant line bundles in nonzero degree.

\begin{lemma}\label{lem62ldd}
If $A:\T\to\SL(2,\R)$ and $E:\T\to\R\mathbb P^1\cong\R/\pi\Z$ are continuous, then $\deg(AE)=\deg E+2\deg A$. In particular, a cocycle $(\alpha,A)$ with $\deg A\ne0$ admits no continuous invariant line bundle $E:\T\to\R\mathbb P^1$.
\end{lemma}
\begin{proof}
Deforming $A(x)$ to $R_{dx}$ ($d=\deg A$) shows that $A$ shifts the projective degree of $E$ by $2d$, whereas $\deg(E\circ T)=\deg E$; invariance $E\circ T=AE$ would thus force $2d=0$.
\end{proof}

\begin{proposition}\label{prop63dss}
Let $\alpha\in\R\setminus\mathbb Q$ and $A\in C^\omega(\T,\SL(2,\R))$ with $d=\deg A>0$, $L(\alpha,A)>0$, and $\omega(\alpha,A)=d$. Suppose there exists a 4-periodic real-analytic gauge $B_0:\R\to\SL(2,\R)$ satisfying \eqref{eq403qrt} such that $C=(B_0\circ T)AB_0^{-1}$ is strictly positively monotonic. Then there exists a 1-periodic gauge $B\in C^\omega(\T,\SL(2,\R))$ such that $(B\circ T^N)A_N B^{-1}$ is strictly positively monotonic on $\T$ for all sufficiently large $N$.
\end{proposition}
\begin{proof}
Equip $\R/4\Z$ with normalized Haar measure $dx/4$ and let $T_4(x)=x+\alpha\pmod 4$. This rotation is ergodic. Since $B_0$ and $B_0^{-1}$ are bounded on the fourfold cover, $C$ has Lyapunov exponent $L(\alpha,A)>0$ there. For $0<\Im z<h$ sufficiently small, $\widehat C(z)$ strictly contracts $\overline\D$, with holomorphic unstable and stable sections $m_u(z)\in\D$ and $m_s(z)\in\widehat\C\setminus\overline\D$. Define
\[
n(z)=\frac1{m_s(z)},\qquad f(z)=n(z)^2, \quad 0<\Im z<h.
\]
Then $n$ is holomorphic with $|n|<1$. By~\eqref{eq403qrt}, $C(z+1)=R_{-1/4}C(z)R_{1/4}$; since $R_{-1/4}$ acts projectively on $\widehat\C$ as multiplication by $-1$, uniqueness of the stable section forces $m_s(z+1)=-m_s(z)$, and hence
\[
n(z+1)=-n(z),\qquad f(z+1)=f(z).
\]
Thus $f$ is 1-periodic on $0<\Im z<h$ with $|f|<1$.

We verify that $f$ satisfies the hypotheses of Lemma~\ref{lem51mjj}. The boundary-modulus condition follows from positive Lyapunov exponent; nonzero degree gives the obstruction to continuation across the real circle. By Fatou's theorem, the nontangential boundary limit $n(x)=\lim_{y\downarrow0}n(x+iy)$ exists for almost every $x\in\R/4\Z$ and defines a measurable $C$-invariant section $m_s(x)=1/n(x)$. If $|n(x)|<1$ on a set of positive measure, ergodicity of $T_4$ implies $|n(x)|<1$ almost everywhere. In Cayley coordinates, complex conjugation on $\C^2$ corresponds to the reflection $\iota(\zeta)=1/\overline\zeta$, which commutes with $\widehat C(x)\in\SU(1,1)$. Hence $\overline{n(x)}=\iota(m_s(x))\in\D$ defines a measurable $C$-invariant section inside $\D$, which under the inverse Cayley transform corresponds to the measurable $C$-invariant section
\[
w(x)=i\frac{1-\overline{n(x)}}{1+\overline{n(x)}}=u(x)+iv(x)
\]
in the upper half-plane $\{v>0\}$. As in the proof of Lemma~\ref{lem302rkk}, the symmetric positive-definite matrix
\[
M(x)=\frac1{v(x)}\begin{pmatrix}1&u(x)\\u(x)&u(x)^2+v(x)^2\end{pmatrix}\in\SL(2,\R)
\]
satisfies $M(T_4x)=C(x)M(x)C(x)^T$. Indeed, the matrix
\[
V(x)=v(x)^{-1/2}\begin{pmatrix}1&0\\u(x)&v(x)\end{pmatrix}
\]
sends the line spanned by $(1,i)^T$ to the line spanned by $(1,w(x))^T$ and satisfies $M=VV^T$. Invariance of $w$ implies $V(T_4x)^{-1}C(x)V(x)\in\SO(2,\R)$, proving the stated covariance. Hence $G=M^{-1/2}$ also measurably conjugates $C$ into $\SO(2,\R)$.

Since $G$ is a measurable $\SL(2,\R)$-valued map, we may choose a set $E$ of positive measure on which $G$ and $G^{-1}$ are bounded. Almost every point of $E$ is both recurrent to $E$ and Oseledets regular for the normalized Haar measure. At such a point $x$, along its return times $n_k\to\infty$,
\[
C_{n_k}(x)=G(T_4^{n_k}x)^{-1}O_{n_k}(x)G(x),
\qquad O_{n_k}(x)\in\SO(2,\R),
\]
so $\|C_{n_k}(x)\|$ remains bounded, contradicting the positive top Lyapunov exponent at $x$. Therefore $|n(x)|=|f(x)|=1$ almost everywhere on $\T$.

Next we check that $f$ cannot extend holomorphically across all of $\T$. Indeed, if it did, then $|f|=1$ everywhere on $\T$ by continuity, so $f$ would be nonvanishing on a thin strip $|\Im z|<\delta_0$ around $\R$. On this simply connected strip, $f$ admits a holomorphic square root agreeing with $n$ on $0<\Im z<\delta_0$, which extends $n$ holomorphically across $\R$ while preserving $n(z+1)=-n(z)$ by the identity theorem. On $\R$, $|n(x)|=1$ then defines a continuous real $C$-invariant line bundle $\ell_s(x)$ satisfying $\ell_s(x+1)=R_{-1/4}\ell_s(x)$. Pulling back by $B_0$ gives a continuous $A$-invariant line bundle $E(x)=B_0(x)^{-1}\ell_s(x)$ satisfying
\[
E(x+1)=B_0(x)^{-1}R_{1/4}R_{-1/4}\ell_s(x)=E(x),
\]
which descends to a continuous real invariant line bundle for $(\alpha,A)$ on $\T$, contradicting Lemma~\ref{lem62ldd}.

We may therefore apply Lemma~\ref{lem51mjj} to $f$, obtaining a Blaschke factor $\beta$ satisfying \eqref{eq502mjj} on $0<y\le\delta/2$. Fix $\eps\in(0,\delta/2]$ small enough to lie in the affine interval~\eqref{eq208aff} and in the analyticity strip of $B_0$. On a simply connected open strip in the universal cover containing $|\Im z|\le\eps$, choose a holomorphic logarithm $\Log\beta(z)$ and define
\[
F(z)=\frac{1}{8\pi i}\Log\beta(z),\qquad
B(z)=R_{F(z)}B_0(z).
\]
Since $|\beta(x)|=1$ and $\wind\beta=1$ on $\T$, the chosen logarithm satisfies
\[
\Log\beta(z+1)-\Log\beta(z)=2\pi i.
\]
Thus $F$ is real-valued on $\R$ and satisfies $F(z+1)=F(z)+1/4$, which gives
\[
B(z+1)=R_{1/4}R_{F(z)}R_{-1/4}B_0(z)=B(z),
\]
so $B\in C^\omega(\T,\SL(2,\R))$ is 1-periodic. Under the Cayley transform, $R_{F(z)}$ acts projectively as multiplication by
\[
q(z)=e^{4\pi iF(z)}=\exp\bigl(\tfrac12\Log\beta(z)\bigr).
\]
At height $\eps$, the unstable section $q m_u$ of $D=(B\circ T)AB^{-1}$ remains in $\D$, since $|q|=|\beta|^{1/2}<1$ and $|m_u|<1$. For the stable section, if $n(x+i\eps)=0$, then $m_s(x+i\eps)=\infty$ and its image is again~$\infty$. At all other points,~\eqref{eq502mjj} gives
\[
|q(x+i\eps)m_s(x+i\eps)|
=\frac{|\beta(x+i\eps)|^{1/2}}{|n(x+i\eps)|}
=\left(\frac{|\beta(x+i\eps)|}{|f(x+i\eps)|}\right)^{1/2}>1.
\]
Thus the stable section remains in $\widehat\C\setminus\overline\D$ everywhere on the upper circle.
Since $B$ is a 1-periodic conjugacy holomorphic on a neighborhood of the closed strip $|\Im z|\le\eps$, $D$ has degree $d$, and $D_\eps$ is uniformly hyperbolic with acceleration $\omega(\alpha,D_\eps)=\omega(\alpha,A_\eps)=d$. Lemma~\ref{lem601ohh} now implies that $D_N=(B\circ T^N)A_N B^{-1}$ is strictly positively monotonic on $\T$ for all sufficiently large $N$.
\end{proof}

\begin{proof}[Proof of Theorem~\ref{thm11mnn}]
Assume first that $\omega(\alpha,A)=|d|$. If $L(\alpha,A)=0$, premonotonicity follows from Proposition~\ref{prop25zrr}. If $L(\alpha,A)>0$ and $d>0$, combining Proposition~\ref{prop32cnnc}, Lemma~\ref{lem402qtt}, and Proposition~\ref{prop63dss} yields a 1-periodic gauge $B\in C^\omega(\T,\SL(2,\R))$ such that $(B\circ T^N)A_N B^{-1}$ is strictly positively monotonic for large $N$. If $L(\alpha,A)>0$ and $d<0$, applying the positive-degree case to the reflected cocycle $(-\alpha,A(-\cdot))$ (which has degree $-d>0$, Lyapunov exponent $L(\alpha,A)$, and acceleration $\omega(\alpha,A)$ since $L(\alpha,A_y)=L(\alpha,A_{-y})$) and reflecting the resulting gauge proves premonotonicity with negative sign.

Conversely, suppose $(\alpha,A)$ is premonotonic. Approximating the $C^1$ gauge from the definition of premonotonicity by a $C^1$-close real-analytic map $B\in C^\omega(\T,\SL(2,\R))$, the map $C(x)=B(x+N\alpha)A_N(x)B(x)^{-1}$ is strictly monotonic of degree $Nd$ over the irrational rotation $\beta=N\alpha$. If $d>0$, then $Q_C>0$, so $\widehat C_y$ strictly contracts $\overline\D$ for every sufficiently small $y>0$. At height $y$, its unstable section $(1,m_y)^T$ ($|m_y|<1$) has multiplier $\lambda_y=a_y+b_y m_y$. Since $|b_y/a_y|<1$ on a collar of the real circle, $\wind\lambda_y=\wind a_y=-\deg C=-Nd$. Formula~\eqref{eq206wdll}, applied at each such positive height, therefore gives
\[
\omega(\beta,C_y)=Nd\qquad (0<y<\eps)
\]
for some $\eps>0$. Thus the finite convex function $\ell(y)=L(\beta,C_y)$ has constant right derivative $2\pi Nd$ on $(0,\eps)$. Its continuity at $0$ gives
\[
L(\beta,C_y)=L(\beta,C)+2\pi Nd\,y
\qquad (0\le y<\eps),
\]
and hence $\omega(\beta,C)=Nd$. Conjugacy invariance and~\eqref{eq27itrr} now yield
\[
N\omega(\alpha,A)=\omega(N\alpha,A_N)=\omega(N\alpha,C)=Nd,
\]
and reflection gives $\omega(\alpha,A)=|d|$ when $d<0$.
\end{proof}

Premonotonicity is open in $(\alpha,A)$ for the $C^1$ topology. Theorem~\ref{thm11mnn} therefore gives the following corollary.

\begin{corollary}
Let $\alpha\in\R\setminus\mathbb Q$ and $A\in C^\omega(\T,\SL(2,\R))$ with $d=\deg A\ne0$ and $\omega(\alpha,A)=|d|$. Then there exists a $C^1$-neighborhood $\mathcal V\subset\R\times C^1(\T,\SL(2,\R))$ of $(\alpha,A)$ such that every real-analytic cocycle $(\widetilde\alpha,\widetilde A)\in\mathcal V$ with $\widetilde\alpha\in\R\setminus\mathbb Q$ satisfies $\omega(\widetilde\alpha,\widetilde A)=|d|$.
\end{corollary}

\section{A variational formula for signedness}\label{sec70vrr}

To prove Proposition~\ref{prop103crr}, we use the Hahn--Banach alternative of Lemma~\ref{lem31hbb} to derive a variational formula for the margin in \eqref{eq104psii}. For $A\in C(\T,\SL(2,\R))$ and $\alpha\in\R$, we use the full notation $\TT_{\alpha,A}$ for the pullback operator and set $m(K)=\min_{x\in\T,\,\|v\|=1}v^TK(x)v$ for $K\in C(\T,\Sym_2\R)$. Let $\mathfrak M_{\alpha,A}$ denote the (possibly empty) weak-$*$ compact convex set of positive $\Sym_2\R$-valued Borel measures $\mu$ on $\T$ satisfying
\begin{equation}\label{eq71vrmm}
\tr\mu(\T)=1,\qquad
\int_\T\tr\bigl((\TT_{\alpha,A}H-H)\,d\mu\bigr)=0
\quad\text{for all }H\in C(\T,\Sym_2\R).
\end{equation}
For $A\in C^1(\T,\SL(2,\R))$, define the signedness margin
\[
\chi(\alpha,A)=\sup_{H\in C(\T,\Sym_2\R)}
m(Q_A+\TT_{\alpha,A}H-H).
\]

\begin{proposition}\label{prop701mgg}
For every $A\in C^1(\T,\SL(2,\R))$ and $\alpha\in\R$,
\begin{equation}\label{eq703mxmn}
\chi(\alpha,A)=\inf_{\mu\in\mathfrak M_{\alpha,A}}
\int_\T\tr(Q_A\,d\mu),
\end{equation}
with the convention $\inf\varnothing=+\infty$. When $\mathfrak M_{\alpha,A}\ne\varnothing$, the infimum is finite and attained.
\end{proposition}
\begin{proof}
For any $H\in C(\T,\Sym_2\R)$ and $\mu\in\mathfrak M_{\alpha,A}$, integrating $Q_A+\TT_{\alpha,A}H-H\ge m(Q_A+\TT_{\alpha,A}H-H)I$ against $\mu$ gives $m(Q_A+\TT_{\alpha,A}H-H)\le\int_\T\tr(Q_A\,d\mu)$. Conversely, for any $t\in\R$ with $t>\chi(\alpha,A)$, there is no $H\in C(\T,\Sym_2\R)$ satisfying $(Q_A-tI)+\TT_{\alpha,A}H-H>0$. Lemma~\ref{lem31hbb} applied to $F=Q_A-tI$ then produces a nonzero invariant positive matrix measure, which after trace normalization lies in $\mathfrak M_{\alpha,A}$ and satisfies $\int_\T\tr(Q_A\,d\mu)\le t$. This establishes~\eqref{eq703mxmn}.

Weak-$*$ compactness of $\mathfrak M_{\alpha,A}$ follows from the Banach--Alaoglu theorem (as the corresponding normalized positive functionals $\ell$ on $C(\T,\Sym_2\R)$ satisfy $\|\ell\|=\ell(I)=1$), so the infimum is attained whenever $\mathfrak M_{\alpha,A}\ne\varnothing$.
\end{proof}

\begin{corollary}\label{cor72scc}
For $A\in C^1(\T,\SL(2,\R))$ and $\alpha\in\R$, the following are equivalent:
\begin{enumerate}
\item $\chi(\alpha,A)>0$;
\item there exists $H\in C(\T,\Sym_2\R)$ (or a trigonometric polynomial $H$) with $Q_A+\TT_{\alpha,A}H-H>0$;
\item every $\mu\in\mathfrak M_{\alpha,A}$ satisfies $\int_\T\tr(Q_A\,d\mu)>0$.
\end{enumerate}
\end{corollary}
\begin{proof}
If $\mathfrak M_{\alpha,A}$ is nonempty, compactness and continuity of the pairing imply that condition~(3) is equivalent to positivity of its minimum. Proposition~\ref{prop701mgg} therefore gives $(1)\Leftrightarrow(3)$; if $\mathfrak M_{\alpha,A}$ is empty, both hold under the convention $\inf\varnothing=+\infty$. The equivalence $(1)\Leftrightarrow(2)$ follows from the definition of $\chi$, and uniform approximation by real symmetric trigonometric polynomials preserves the strict inequality.
\end{proof}

\begin{proof}[Proof of Proposition~\ref{prop103crr}]
Assume $d>0$; the case $d<0$ follows by applying the positive-degree case to $(-\alpha,A(-\cdot))$. By Corollary~\ref{cor72scc}, conditions~(2) and~(3) are equivalent. If $\omega(\alpha,A)=d$, Theorem~\ref{thm11mnn} provides $B\in C^\omega(\T,\SL(2,\R))$ and $N\ge1$ such that $C=(B\circ T^N)A_N B^{-1}$ satisfies $Q_C>0$. Set $H_B=B^T(-J)B'$ and recall the affine map $\FF(H)=Q_A+\TT_{\alpha,A}H$. Lemma~\ref{lem41gga} and~\eqref{eq21pqq} give $\FF^N(H_B)-H_B=B^T Q_C B>0$. The averaged analytic field $H=\frac1N\sum_{j=0}^{N-1}\FF^j(H_B)$ then satisfies $\FF(H)-H=\frac1N(\FF^N(H_B)-H_B)>0$, proving $(1)\Rightarrow(2)$.

Conversely, if $H\in C^\omega(\T,\Sym_2\R)$ satisfies~\eqref{eq104psii}, Lemma~\ref{lem402qtt} yields a 4-periodic analytic gauge $B_0$ such that $C=(B_0\circ T)AB_0^{-1}$ is strictly positively monotonic on $\R/4\Z$ with degree $4d$. Rescaling $x=4u$ gives a 1-periodic strictly positively monotonic cocycle $C(4\cdot)$ over $\alpha/4$ of degree $4d$. The covering identities in Section~\ref{sec23hyp} and the computation of the winding number of the unstable multiplier from the proof of Theorem~\ref{thm11mnn} give
\[
4\omega(\alpha,A)=\omega(\alpha/4,A(4\cdot))=\omega(\alpha/4,C(4\cdot))=4d,
\]
proving $(2)\Rightarrow(1)$.
\end{proof}

\section{Equivariant deformation and path connectedness}\label{sec8dff}

In $\Sp(2,\R)=\SL(2,\R)$, strict positive monotonicity $Q_C=C^T(-J)C'>0$ is equivalent to the positive-path condition $C'=JPC$ with $P>0$ \cite{Sli01,WZ24}. To deform elements of $\mathcal M_d(\alpha)$ while preserving 1-periodicity on $\T$, we use an explicit $\SO(2,\R)$-equivariant retraction onto $\SO(2,\R)$ that commutes with the quarter-turn monodromy of Lemma~\ref{lem402qtt}.

\subsection{An equivariant contraction of monotonic loops}

In this subsection we identify $(v_1,v_2)\in\R^2$ with $v_1+iv_2\in\C$ and use $a,b$ for the complex coefficients in the representation $Av=av+b\overline v$. These coefficients are distinct from the Cayley entries in \eqref{eq202cyy}. Every $\R$-linear map $A\in\SL(2,\R)$ has a unique such representation, with $|a|^2-|b|^2=1$. In particular, $a\ne0$, and the pair $(u,w)\in\partial\D\times\D$ gives global coordinates:
\[
u=\frac a{|a|}\in\partial\D,\qquad w=\frac ba\in\D,\qquad
a=\frac{u}{\sqrt{1-|w|^2}},\qquad
b=\frac{uw}{\sqrt{1-|w|^2}}.
\]
For $t\in[0,1]$, define $\mathcal P_t(A)\in\SL(2,\R)$ by the coefficients
\begin{equation}\label{eq82cntt}
a_t=\frac{u}{\sqrt{1-t^2|w|^2}},\qquad
b_t=\frac{tu w}{\sqrt{1-t^2|w|^2}},
\end{equation}
so that $\mathcal P_1(A)=A$, $\mathcal P_0(A)$ is the rotation $v\mapsto uv$, and $\mathcal P_t$ fixes $\SO(2,\R)$ pointwise.

\begin{lemma}\label{lem81cttr}
If $A(x)$ is a strictly positively (resp.\ negatively) monotonic $C^1$ curve in $\SL(2,\R)$, then $\mathcal P_t(A(x))$ is strictly positively (resp.\ negatively) monotonic for every $t\in[0,1]$. Moreover, for every $U\in\SO(2,\R)$,
\begin{equation}\label{eq803eqvv}
\mathcal P_t(UAU^{-1})=U\mathcal P_t(A)U^{-1}.
\end{equation}
If $A$ is real analytic and periodic, then $(t,x)\mapsto\mathcal P_t(A(x))$ is real analytic and periodic in $x$. Moreover, there exist $h>0$ and an open neighborhood $\Omega$ of the closed strip $\{|\Im z|\le h\}$ such that, for every $t\in[0,1]$, the map $x\mapsto\mathcal P_t(A(x))$ extends holomorphically to $\Omega$.
\end{lemma}
\begin{proof}
Write $u(x)=e^{i\varphi(x)}$ locally. For any nonzero $v\in\C$, setting $\zeta=v/\overline v\in\partial\D$ gives
\[
\frac d{dx}\arg(A(x)v)
=\varphi'(x)+\Im\frac{w'(x)}{\zeta+w(x)}.
\]
As $\zeta$ ranges over $\partial\D$, $(\zeta+w)^{-1}$ traces the circle with center $-\overline w/(1-|w|^2)$ and radius $(1-|w|^2)^{-1}$. Thus strict positive monotonicity of $\mathcal P_t(A)$ at $x$ is equivalent to $\varphi'(x)>F(t)$, where
\[
F(t)=\frac{t|w'|+t^2\Im(\overline w w')}{1-t^2|w|^2}.
\]
Since $|\Im(\overline w w')|\le|w|\,|w'|$, differentiating $F$ with respect to $t\in[0,1]$ yields
\[
\begin{aligned}
F'(t)&=\frac{|w'|(1+t^2|w|^2)+2t\Im(\overline w w')}{(1-t^2|w|^2)^2}\\
&\ge\frac{|w'|(1-t|w|)^2}{(1-t^2|w|^2)^2}\ge0.
\end{aligned}
\]
Hence $\varphi'>F(1)$ implies $\varphi'>F(t)$ for all $t\in[0,1]$. The negative case follows by reversing the orientation of $x$. If $U\in\SO(2,\R)$ acts by $v\mapsto e^{i\gamma}v$, then $UAU^{-1}$ has coefficients $(a,e^{2i\gamma}b)$ and coordinates $(u,e^{2i\gamma}w)$, which immediately gives \eqref{eq803eqvv}. Finally, write $f^*(z)=\overline{f(\overline z)}$ for the Schwarz reflection. For periodic analytic $A$, compactness gives a common neighborhood of a thin closed strip $|\Im z|\le h$ on which $aa^*$ and $1-t^2ww^*$ take values in the right half-plane for $t$ in a complex neighborhood of $[0,1]$. The principal square roots in~\eqref{eq82cntt}, together with the reflected coefficients $u^*/\sqrt{1-t^2ww^*}$ and $tu^*w^*/\sqrt{1-t^2ww^*}$, then define an $\SL(2,\C)$-valued map jointly holomorphic in $(t,z)$ and periodic in $z$.
\end{proof}

\subsection{Paths in the minimal-acceleration locus}

\begin{proof}[Proof of Theorem~\ref{thm14ptth}]
Assume first that $d>0$. For $A\in\mathcal M_d(\alpha)$, Proposition~\ref{prop103crr} and Lemma~\ref{lem402qtt} give a 4-periodic analytic gauge $B_0$ and a strictly positively monotonic cocycle $C=(B_0\circ T)AB_0^{-1}$ satisfying $B_0(x+1)=MB_0(x)$ and $C(x+1)=MC(x)M^{-1}$, where $M=R_{-1/4}$. We join $A$ to $R_{dx}$ by concatenating three paths:

\medskip
\noindent{\it Step 1. Contract the monotonic cocycle.}
For $t\in[0,1]$, define
\[
C_t=\mathcal P_t(C),\qquad
\mathcal A_t(x)=B_0(x+\alpha)^{-1}C_t(x)B_0(x).
\]
Then $\mathcal A_1=A$. By Lemma~\ref{lem81cttr}, $C_t$ is strictly positively monotonic and satisfies $C_t(x+1)=MC_t(x)M^{-1}$, so the monodromy cancels:
\[
\mathcal A_t(x+1)=B_0(x+\alpha)^{-1}M^{-1}\bigl(MC_t(x)M^{-1}\bigr)MB_0(x)=\mathcal A_t(x).
\]
Thus $\mathcal A_t\in C^\omega(\T,\SL(2,\R))$ has degree $d$, and the covering identities in Section~\ref{sec23hyp} and the computation of the winding number of the unstable multiplier from the proof of Theorem~\ref{thm11mnn} give $4\omega(\alpha,\mathcal A_t)=\omega(\alpha/4,C_t(4\cdot))=4d$, so $\mathcal A_t\in\mathcal M_d(\alpha)$ for all $t\in[0,1]$.

\medskip
\noindent{\it Step 2. Deform the periodic gauge.}
At $t=0$, $C_0$ takes values in $\SO(2,\R)$, which is abelian, so $C_0(x+1)=MC_0(x)M^{-1}=C_0(x)$ is already 1-periodic. Set $K(x)=R_{x/4}B_0(x)$. Since $R_{(x+1)/4}M=R_{x/4}$, $K\in C^\omega(\T,\SL(2,\R))$ is 1-periodic, and
\[
D(x):=K(x+\alpha)\mathcal A_0(x)K(x)^{-1}=R_{\alpha/4}C_0(x)
\]
is a 1-periodic $\SO(2,\R)$-valued cocycle of degree $d$. For $u\in[0,1]$, set $K_u=\mathcal P_u(K)$ and
\[
\mathcal B_u(x)=K_u(x+\alpha)^{-1}D(x)K_u(x).
\]
Then $\mathcal B_1=\mathcal A_0$, $\mathcal B_0$ takes values in $\SO(2,\R)$, and each $\mathcal B_u$ is analytically conjugate on $\T$ to the degree-$d$ rotation cocycle $D$. Writing $D(x)=R_{dx+\psi(x)}$ with $\psi\in C^\omega(\T,\R)$, its Cayley eigenvalues at height $y$ are $e^{\pm 2\pi i(dx+idy+\psi(x+iy))}$, so ergodic averaging gives
\[
L(\alpha,D_y)=2\pi\left|\int_\T\Im\bigl(dx+idy+\psi(x+iy)\bigr)\,dx\right|=2\pi dy, \quad 0\le y<h,
\]
where $h>0$ is a common strip width for the paths above and $\int_\T\psi(x+iy)\,dx=\int_\T\psi(x)\,dx\in\R$. Thus $\omega(\alpha,\mathcal B_u)=\omega(\alpha,D)=d$, and $\mathcal B_u\in\mathcal M_d(\alpha)$ for all $u\in[0,1]$.

\medskip
\noindent{\it Step 3. Remove the periodic part of the angle.}
Since $\mathcal B_0(x)=R_{dx+\psi_0(x)}$ for some $\psi_0\in C^\omega(\T,\R)$, the path $v\mapsto R_{dx+(1-v)\psi_0(x)}$ ($v\in[0,1]$) joins $\mathcal B_0$ to $R_{dx}$ inside $\mathcal M_d(\alpha)$. Lemma~\ref{lem81cttr} and the explicit paths above give a common smaller closed strip on which all three segments are continuous in the supremum norm. Concatenating the three segments completes the proof for $d>0$, and the case $d<0$ follows by reflecting $x\mapsto-x$ and $\alpha\mapsto-\alpha$.
\end{proof}

\begin{corollary}\label{cor803crr}
For every integer $r\ge1$, $\alpha\in\R\setminus\mathbb Q$, and $d\in\Z\setminus\{0\}$, the degree-$d$ premonotonic locus in $C^r(\T,\SL(2,\R))$ is path connected in the $C^r$ topology.
\end{corollary}
\begin{proof}
Fix a cocycle $A$ in this locus. Choose a real matrix-valued trigonometric polynomial $P$ sufficiently close to $A$ in $C^r$. Then $\det P>0$ on $\T$, and $\widetilde A=P/\sqrt{\det P}$ is a real-analytic $\SL(2,\R)$-valued approximation of $A$. Since premonotonicity is $C^1$-open, choosing $P$ sufficiently close ensures that the normalized interpolation
\[
A_s=\frac{P_s}{\sqrt{\det P_s}},\qquad
P_s=(1-s)A+s\widetilde A\qquad (s\in[0,1])
\]
is well defined and remains in the degree-$d$ premonotonic locus. This path is continuous in $C^r$. By Theorem~\ref{thm11mnn}, $\widetilde A\in\mathcal M_d(\alpha)$. Theorem~\ref{thm14ptth} connects $\widetilde A$ to $R_{dx}$ through a path in $\mathcal M_d(\alpha)$ analytic on a common strip. This path consists of premonotonic cocycles by Theorem~\ref{thm11mnn} and is continuous in $C^r$ by Cauchy estimates on a thinner strip.
\end{proof}

\section{Regularity and the role of degree}\label{sec9mchh}

To compare Avila's degree-zero regularity theorem \cite[Theorem~6]{Av15} directly with Theorem~\ref{thm11mnn}, for $\alpha\in\R\setminus\mathbb Q$ and $A\in C^\omega(\T,\SL(2,\R))$ of degree $d\in\Z$ we set
\[
\mathcal E_A(y)=L(\alpha,A_y)-L(\alpha,A)-2\pi|d|\,|y|,
\qquad
\omega_{\mathrm{ex}}(\alpha,A)=\omega(\alpha,A)-|d|.
\]

\begin{proposition}\label{prop91exx}
If $A$ extends holomorphically to $|\Im z|<h$, then $\mathcal E_A$ is even, convex, and nonnegative on $(-h,h)$, and $\omega_{\mathrm{ex}}(\alpha,A)\in\Z_{\ge0}$. Moreover, $\omega_{\mathrm{ex}}(\alpha,A)=0$ if and only if $\mathcal E_A(y)=0$ for all sufficiently small $|y|$.
\end{proposition}
\begin{proof}
Evenness of $y\mapsto L(\alpha,A_y)$ and Proposition~\ref{prop21lww} imply that $\mathcal E_A$ is even and nonnegative, with right derivative $2\pi\omega_{\mathrm{ex}}(\alpha,A)\ge0$ at $y=0$. Convexity on $[0,h)$ and $(-h,0]$, together with $\mathcal E_A'(0^+)\ge0\ge\mathcal E_A'(0^-)$, ensures convexity on $(-h,h)$. Quantization of acceleration gives $\omega_{\mathrm{ex}}(\alpha,A)\in\Z_{\ge0}$ and, together with right continuity of the right derivative, implies that $\mathcal E'_{A,+}$ is constant on $[0,\eta]$ for some $\eta\in(0,h)$. Since $\mathcal E_A(0)=0$, this constancy proves the final equivalence.
\end{proof}

\begin{corollary}\label{cor902exx}
Suppose $L(\alpha,A)>0$. Then $\omega_{\mathrm{ex}}(\alpha,A)=0$ is equivalent to uniform hyperbolicity when $d=0$, and to premonotonicity when $d\ne0$.
\end{corollary}
\begin{proof}
When $d=0$, vanishing of $\mathcal E_A$ near $0$ is precisely Avila's regularity condition, which at $L(\alpha,A)>0$ is equivalent to uniform hyperbolicity \cite[Theorem~6]{Av15}; when $d\ne0$, the equivalence follows from Proposition~\ref{prop91exx} and Theorem~\ref{thm11mnn}.
\end{proof}

The following criterion characterizes uniform hyperbolicity by the absence of nonzero finite invariant positive matrix measures. Its proof uses quadratic cones; compare \cite{Woj85}.

\begin{proposition}\label{prop902ehh}
For any $\alpha\in\R$ and $A\in C(\T,\SL(2,\R))$, the following are equivalent:
\begin{enumerate}
\item $(\alpha,A)$ is uniformly hyperbolic;
\item $\mathfrak M_{\alpha,A}=\varnothing$;
\item there exists $H\in C(\T,\Sym_2\R)$ such that $A^T(H\circ T_\alpha)A-H>0$ on $\T$.
\end{enumerate}
If $A\in C^1(\T,\SL(2,\R))$, these are also equivalent to $\chi(\alpha,A)=+\infty$. In particular, $\mathfrak M_{\alpha,A}\ne\varnothing$ whenever $\deg A\ne0$.
\end{proposition}
\begin{proof}
Applying Lemma~\ref{lem31hbb} with $F=0$ gives $(2)\Leftrightarrow(3)$, and Proposition~\ref{prop701mgg} gives the equivalence with $\chi(\alpha,A)=+\infty$ for $C^1$ cocycles.

To see that $(3)\Rightarrow(1)$, write $h_x(v)=v^TH(x)v$ and choose $\kappa>0$ such that
\begin{equation}\label{eq91stqq}
h_{T_\alpha x}(A(x)v)-h_x(v)\ge\kappa\|v\|^2, \quad x\in\T,\ v\in\R^2.
\end{equation}
Set $M=\|H\|_\infty>0$. We first show that $H(x)$ has signature $(1,1)$ at every $x\in\T$. Suppose that $H(x)\ge0$ at some point, and define
\[
G_n=A_n(x)^TH(T_\alpha^n x)A_n(x), \quad n\ge1.
\]
By~\eqref{eq91stqq}, $G_1\ge\kappa I$, and for $n\ge2$,
\[
G_n-G_{n-1}\ge\kappa A_{n-1}(x)^TA_{n-1}(x)
\ge\frac{\kappa}{M}G_{n-1}.
\]
By induction,
\[
G_n\ge\kappa\left(1+\frac{\kappa}{M}\right)^{n-1}I.
\]
Since $\det A_n=1$, this contradicts the boundedness of $\det G_n=\det H(T_\alpha^n x)$. Applying the same argument to the inverse cocycle with $-H$ excludes $H(x)\le0$.

Thus the positive cones $C_x^+=\{[v]\in\R\mathbb P^1:h_x(v)\ge0\}$ form a continuous family of closed projective intervals mapped strictly inside $\operatorname{int}(C_{T_\alpha x}^+)$ by $A(x)$. Moreover,~\eqref{eq91stqq} gives the uniform growth bound
\[
\|A_n(x)v\|^2\ge\frac{\kappa}{M}
\left(1+\frac{\kappa}{M}\right)^{n-1}\|v\|^2
\qquad (h_x(v)\ge0).
\]
Set $C_x^-=\{[v]\in\R\mathbb P^1:h_x(v)\le0\}$. Applying strict cone contraction to $A$ on $C^+$ and to the inverse cocycle on $C^-$ produces continuous invariant line bundles
\[
\begin{gathered}
E^u(x)=\bigcap_{n\ge1}A_n(T_\alpha^{-n}x)C_{T_\alpha^{-n}x}^+,\\
E^s(x)=\bigcap_{n\ge1}A_n(x)^{-1}C_{T_\alpha^n x}^-.
\end{gathered}
\]
The displayed growth bound on $C^+$ and its counterpart for the inverse cocycle on $C^-$ give a uniform exponential dichotomy.

Conversely, if $(\alpha,A)$ is uniformly hyperbolic with continuous invariant splitting $v=v_u+v_s$, setting $A_{-k}(x)=A_k(T_\alpha^{-k}x)^{-1}$ and defining
\[
h_x(v)=\sum_{k=1}^\infty\|A_{-k}(x)v_u\|^2-\sum_{k=0}^\infty\|A_k(x)v_s\|^2
\]
yields a continuous symmetric quadratic form satisfying $h_{T_\alpha x}(A(x)v)-h_x(v)=\|v_u\|^2+\|v_s\|^2\ge\frac12\|v\|^2$, proving $(1)\Rightarrow(3)$. Finally, when $\deg A\ne0$, Lemma~\ref{lem62ldd} rules out uniform hyperbolicity, so $\mathfrak M_{\alpha,A}\ne\varnothing$.
\end{proof}

For irrational $\alpha$ and real-analytic $A$, Corollary~\ref{cor902exx} and Propositions~\ref{prop103crr}, \ref{prop701mgg}, and~\ref{prop902ehh} imply that, when $d=0$ and $L(\alpha,A)>0$, vanishing excess acceleration $\omega_{\mathrm{ex}}(\alpha,A)=0$ is equivalent to $\mathfrak M_{\alpha,A}=\varnothing$, or equivalently to $\chi(\alpha,A)=+\infty$. When $d\ne0$, topological winding forces $\mathfrak M_{\alpha,A}\ne\varnothing$, and $\omega_{\mathrm{ex}}(\alpha,A)=0$ holds if and only if
\[
\operatorname{sgn}(d)\int_\T\tr(Q_A\,d\mu)>0 \quad \text{for all } \mu\in\mathfrak M_{\alpha,A}.
\]
This condition is equivalent to premonotonicity of $(\alpha,A)$. For $d>0$, it is also equivalent to $0<\chi(\alpha,A)<+\infty$.

\begin{remark}
Neither $L(\alpha,A)>0$ in the $d=0$ criterion nor $d\ne0$ in Theorem~\ref{thm11mnn} can be omitted: the constant cocycle $A\equiv J=R_{1/4}$ has $d=L=\omega=0$ without being uniformly hyperbolic, whereas $A\equiv I$ has $d=\omega=0$ without being premonotonic. Likewise, irrationality of $\alpha$ is essential in Theorem~\ref{thm11mnn}: for $\alpha=0$, the rotation loop $A(x)=R_{\phi(x)}$ with $\phi(x)=x-\frac{1}{2\pi}\sin(2\pi x)$ has $\deg A=1$, $L(0,A)=0$, and
\[
\omega(0,A)=\lim_{y\downarrow0}\frac1y\int_\T|\Im\phi(x+iy)|\,dx
=\int_\T(1-\cos 2\pi x)\,dx=1,
\]
yet $A(0)=I$ and $A'(0)=0$, so $(B A^N B^{-1})'(0)=0$ for every $N\ge1$ and every $B\in C^1(\T,\SL(2,\R))$.
\end{remark}

\appendix
\section{Finite Fourier extremality}\label{app10frr}

For trigonometric polynomial cocycles of nonzero degree, equality between the absolute topological degree and the Fourier degree implies strict monotonicity. The argument does not depend on the base frequency.

\begin{proposition}
Let $A:\T\to\SL(2,\R)$ be a trigonometric polynomial with Fourier support in $[-m,m]$ ($m\ge1$). If $\deg A=m$, then $A$ is strictly positively monotonic on $\T$; if $\deg A=-m$, then $A$ is strictly negatively monotonic on $\T$.
\end{proposition}
\begin{proof}
Assume $\deg A=m$. For any $v\in\R^2\setminus\{0\}$, the loop $f_v(x)=(A(x)v)_1+i(A(x)v)_2$ is nonvanishing on $\T$ with $\wind f_v=m$. Writing $z=e^{2\pi ix}$, $P_v(z)=z^m f_v(x)$ is a polynomial in $z$ of degree at most $2m$ with winding $2m$ on $|z|=1$. By the argument principle, $P_v$ has exactly $2m$ zeros $a_1,\dots,a_{2m}$ (counted with multiplicity), all lying strictly inside $\D$. Hence, for $|z|=1$,
\[
\frac{d}{dx}\arg f_v(x)
=2\pi\left(-m+\sum_{j=1}^{2m}\Re\frac{z}{z-a_j}\right)
=\pi\sum_{j=1}^{2m}\frac{1-|a_j|^2}{|z-a_j|^2}>0,
\]
which by \eqref{eq13angg} means $Q_A>0$ on $\T$. The case $\deg A=-m$ follows by reflection.
\end{proof}

\subsection*{Declaration}
The main mathematical ideas and the overall proof strategy were developed
by the authors. During revision of this manuscript, AI tools were used to
assist with checking calculations and polishing the text. The authors are
responsible for the mathematical content, references, and final text.

\end{document}